\documentclass[12pt]{amsart}
\date{29 August 2018}      
\usepackage{hyperref}
\usepackage{latexsym,amsmath,amsfonts,amscd,amssymb,graphics}
\usepackage{bm}
\usepackage[a4paper,margin=1in]{geometry}
\usepackage[all]{xy}
\usepackage[utf8]{inputenc}

\theoremstyle{plain}  
\newtheorem{theorem}{Theorem}[section]
\newtheorem*{theorem*}{Theorem}

\newtheorem{corollary}[theorem]{Corollary}
\newtheorem{lemma}[theorem]{Lemma}
\newtheorem{proposition}[theorem]{Proposition}
\newtheorem{conjecture}[theorem]{Conjecture}

\theoremstyle{definition}
\newtheorem{definition}[theorem]{Definition}

\newtheorem{assumption}[theorem]{Assumption}

\theoremstyle{remark}

\newtheorem{remark}[theorem]{Remark}

\newtheorem*{claim*}{Claim}

\numberwithin{equation}{section}

\newcommand{\st}{\;|\;}

\renewcommand{\leq}{\leqslant}

\renewcommand{\geq}{\geqslant}

\renewcommand{\setminus}{\smallsetminus}

\newcommand{\into}{\hookrightarrow}

\newcommand{\R}{\mathbb{R}}

\newcommand{\Z}{\mathbb{Z}}
\newcommand{\C}{\mathbb{C}}

\newcommand{\cF}{\mathcal{F}}
\newcommand{\cM}{\mathcal{M}}
\newcommand{\cN}{\mathcal{N}}
\newcommand{\cO}{\mathcal{O}}

\newcommand{\cS}{\mathcal{S}}

\newcommand{\xra}{\xrightarrow}

\newcommand{\PGL}{\mathrm{PGL}}
\newcommand{\PSL}{\mathrm{PSL}}

\newcommand{\SL}{\mathrm{SL}}

\newcommand{\GL}{\mathrm{GL}}

\newcommand{\var}{\mathrm{var}}
\newcommand{\Gn}{\Gamma_n}

\newcommand{\pdeg}{\operatorname{pardeg}}
\newcommand{\pmu}{\operatorname{par\mu}}
\newcommand{\Hom}{\operatorname{Hom}}
\newcommand{\PH}{\operatorname{ParHom}}
\newcommand{\SPH}{\operatorname{SParHom}}
\newcommand{\PE}{\operatorname{ParEnd}}
\newcommand{\SPE}{\operatorname{SParEnd}}

\DeclareMathOperator{\Jac}{Jac} 

\DeclareMathOperator{\Ad}{Ad} 
\DeclareMathOperator{\Div}{Div}

\DeclareMathOperator{\divisor}{div}
\DeclareMathOperator{\Prym}{Prym}
\DeclareMathOperator{\Nm}{Nm}
\DeclareMathOperator{\Sym}{Sym}

\DeclareMathOperator{\tr}{tr}
\DeclareMathOperator{\rk}{rk}

\DeclareMathOperator{\End}{End}
 
\DeclareMathOperator{\Id}{Id}

\newcommand{\Pic}{\operatorname{Pic}}

\renewcommand{\a}{\alpha}

\newcommand{\s}{\sigma}

\newcommand{\p}{\phi}
\newcommand{\vp}{\varphi}

\begin{document}

\title[Topological Mirror Symmetry for Parabolic Higgs bundles]{Topological Mirror Symmetry for Parabolic Higgs bundles}

\author[P. B. Gothen]{Peter B. Gothen}
\address{}
\email{}

\author[A. G. Oliveira]{Andr\'e G. Oliveira}
\address{}
\email{}

\thanks{
This work was partially supported by CMUP (UID/MAT/00144/2013) and
the project PTDC/MAT-GEO/2823/2014 funded by FCT (Portugal) with
national funds.  
Second author was also partially supported by the Post-Doctoral fellowship SFRH/BPD/100996/2014 funded by FCT (Portugal) with national funds.
The authors
acknowledge support from U.S. National Science Foundation grants
DMS 1107452, 1107263, 1107367 "RNMS: GEometric structures And
Representation varieties" (the GEAR Network).
%
}


\begin{abstract}
  We prove the topological mirror symmetry conjecture of
  Hausel--Thaddeus \cite{hausel-thaddeus:2001,hausel-thaddeus:2003}
  for the moduli space of strongly parabolic Higgs bundles of rank two and three, with full
  flags, for any generic weights. Although the main theorem is proved only for rank at most three, most of the results are proved for any prime rank.
\end{abstract}

\maketitle



\section{Introduction}


The Hitchin system is an algebraic completely integrable system. Since
it was introduced by Hitchin \cite{hitchin:1987,hitchin:1987b} thirty
years ago, it has been the subject of much interest, and it has turned
out to have profound connections with several other areas of
mathematics. Its basic ingredient is the moduli space
$\mathcal{M}_d(G)$ of $G$-Higgs bundles $(V,\vp)$ of fixed topological
type $d$ on a closed Riemann surface $X$ for a connected complex
reductive group $G$. Here $V$ is a holomorphic principal $G$-bundle on
$X$ and $\vp$ is a holomorphic $1$-form with values in $\Ad(V)$. This moduli space is a holomorphic symplectic
manifold carrying a hyper-K\"ahler metric. The integrable system is
given by the Hitchin map $h\colon\mathcal{M}_d(G)\to\mathcal{A}$, where
the Hitchin base $\mathcal{A}$ is an affine space whose dimension is
half that of $\mathcal{M}_d(G)$ and the components of $h$ are the
coefficients of the characteristic polynomial of $\vp$.

Mirror symmetry for the Hitchin system was introduced in the work of
Hausel and Thaddeus \cite{hausel-thaddeus:2003} (announced in
\cite{hausel-thaddeus:2001}). It involves the Hitchin systems for the
pair of Langlands dual groups $G=\SL(n,\C)$ and
${}^LG=\PGL(n,\C)$. Hausel and Thaddeus proved that the moduli spaces are
mirror partners in the sense of Strominger--Yau--Zaslow (SYZ)
\cite{strominger-yau-zaslow:1993}; since they considered the case when
$n$ and the degree $d=\deg(V)$ are coprime, this requires equipping the
moduli spaces with natural $B$-fields, or gerbes. Hausel and Thaddeus
have also shown that, in the cases $n=2,3$, the moduli spaces satisfy
topological mirror symmetry, which is an identity of suitably defined
stringy $E$-polynomials (these encode stringy Hodge numbers and again
involve the $B$-field). Moreover, they conjectured that this holds for any
$n$ and $d$ with $(n,d)=1$.

It was later proved by Donagi--Pantev \cite{donagi-pantev:2012} that, more
generally, SYZ mirror symmetry is satisfied by the moduli spaces (or
stacks) of $G$-Higgs bundles for any complex reductive group $G$. On
the other hand, a very recent preprint by Groechenig--Wyss--Ziegler
\cite{Groechenig-etal:2017} uses $p$-adic integration to prove
topological mirror symmetry in the case $G=\SL(n,\C)$ for any $n$ and
$d$ with $(n,d)=1$.

{Parabolic Higgs bundles} were introduced by Simpson
\cite{simpson:1990} as the natural objects to consider for extending
non-abelian Hodge theory to punctured Riemann surfaces. They are pairs
consisting of a parabolic bundle $V$, i.e., a vector bundle with
weighted flags in the fibers over fixed marked points in $X$, and a
Higgs field $\vp$ with values in the parabolic endomorphisms of $V$.

The theory of parabolic Higgs bundles is in many ways analogous to that
of usual Higgs bundles. In particular, there is a parabolic version of
the Hitchin system, which goes back to the work of Bottacin
\cite{bottacin:1995}, Markman \cite{markman:1994} and Nasatyr--Steer
\cite{nasatyr-steer:1995}.  Parabolic Higgs bundles have subsequently
been studied by many authors; we merely point to Boden--Yokogawa
\cite{boden-yokogawa:1996} and Logares--Martens
\cite{logares-martens:2010} as convenient references for parabolic
Higgs bundles and the parabolic Hitchin system.  We emphasize that in
this paper we consider exclusively \emph{strongly  parabolic} Higgs
bundles, meaning that the residue of the Higgs field at the marked
points is nilpotent. These provide the most immediate generalization
of the Hitchin system in that their moduli spaces are symplectic
leaves of more general (Poisson) moduli spaces of (non-strongly)
parabolic Higgs bundles.

In the announcement \cite{hausel-thaddeus:2001} Hausel and Thaddeus
also consider the parabolic case, and outline a proof that SYZ mirror
symmetry holds for any $n$.  Moreover, they state that topological
mirror symmetry holds for {parabolic} Higgs bundles in the case
$G=\SL(n,\C)$ with $n=2,3$, and conjecture that it should be true for
any $n$. The main result of the present paper is a proof of this
conjecture for $n=2,3$ (Theorem~\ref{thm:main1} below). For simplicity
we restrict ourselves to the case of full flags, though our
calculations of $E$-polynomials can in fact be carried through in the
general case.

Our proof follows the basic strategy of Hausel and Thaddeus. It rests
on the observation that it suffices to prove that certain
contributions on each side are identical in order to conclude that the
full stringy $E$-polynomials coincide for the $\SL(n,\C)$ and
$\PGL(n,\C)$ moduli spaces.  On the $\PGL(n,\C)$-moduli space, the
relevant contribution to the stringy $E$-polynomial comes from the
fixed loci in the moduli space under the natural action of non-trivial
elements of the group $\Gamma_n$ of
$n$-torsion points of $\Pic^0(X)$. On the $\SL(n,\C)$-moduli space, the
relevant contribution
is the part of the $E$-polynomial which is \emph{not invariant} under
the action of $\Gamma_n$, also known as the \emph{variant part}, and which is determined by the variant part of the $E$-polynomial of certain fixed point subvarieties under the natural $\C^*$-action.

The description of the fixed loci of elements of $\Gamma_n$ is broadly
parallel to that of \cite{hausel-thaddeus:2003} and essentially rests
on the work of Narasimhan--Ramanan \cite{narasimhan-ramanan:1975}.
The result is that the fixed point loci are described in terms of Prym
varieties of unramified covers of $X$ modulo the action of the Galois
group. However, in the parabolic situation, it turns out that 
this action can be absorbed in the parabolic data, and this simplifies
the arguments somewhat compared to the non-parabolic situation. 

On the other hand, the fixed points of the $\C^*$-action are so-called
Hodge bundles. These are Higgs bundles whose underlying vector bundle
has a direct sum decomposition $V=V_1\oplus\dots\oplus V_l$ with
respect to which the Higgs field $\vp$ has weight one. For rank
$n=2,3$, it is known that only fixed loci consisting of Hodge bundles
whose summands are all line bundles contribute to the variant
$E$-polynomial, but the corresponding result for higher prime rank ---
that only $\C^*$-fixed loci of type $(1,1,\dots,1)$ contribute to the
variant $E$-polynomial --- is not known to be true.  This is the only
missing step for generalising our proof to any prime rank $n$, since
our calculations are done for every such $n$ (this is completely
analogous to the non-parabolic case as treated in
\cite{hausel-thaddeus:2003}).

It turns out that the $B$-field does not play a very prominent role in
the parabolic situation. Indeed, for SYZ mirror symmetry to hold in
the strict sense, i.e., without a $B$-field, it is required that there
be a Lagrangian section of the fibration, providing the natural base
points of the abelian varieties which are the fibers of the integrable
system. This is provided by the parabolic version of the Hitchin
section. There is an isomorphism between moduli spaces of parabolic
Higgs bundles for any two different degrees (requiring an adjustment
of the weights), as long as at least one of the flags are full. Using
this isomorphism, one proves that the Hitchin section --- thus the
required Lagrangian section --- exists for any degree; cf.\
Theorem~\ref{thm:SYZmirror}. One thus also expects the usual (i.e.\
without the $B$-field correction) stringy $E$-polynomials to agree,
and this is indeed what we show to be the case.  On the other hand
Biswas--Dey \cite{biswas-dey:2012} have proved that the moduli spaces
of parabolic Higgs bundles satisfy SYZ mirror symmetry also when
equipped with the natural $B$-fields (analogously to the non-parabolic
case considered in \cite{hausel-thaddeus:2003}).  Thus one would
expect topological mirror symmetry to hold also for the $B$-field
twisted stringy $E$-polynomials and, indeed, calculations analogous to
those carried out here and in \cite{hausel-thaddeus:2003} indicate
that this is the case.

We assume throughout that $g\geq 2$. However, we note that for
$g=0$ the $E$-polynomials trivially agree since the moduli spaces are
identical. For $g=1$ we have chosen not to include the calculation of
 $E$-polynomials, both in order to avoid considering special cases and also
because on the SYZ side mirror symmetry is currently known only for
$g\geq2$ (since otherwise the generic fibre of the Hitchin map may be
singular, see Section~\ref{subsec:Hitmap}).


Here is an outline of the contents of the
paper. Section~\ref{sec:parabolic-higgs-bundles-and-moduli} reviews
basic facts about parabolic Higgs bundles and their moduli. We also
recall how in the parabolic setting moduli spaces for different
degrees $d$ are isomorphic (with a change in parabolic weights). In
Section~\ref{sec:mirror-symmetry} we recall the SYZ mirror symmetry
result for the parabolic Hitchin system, review the stringy
$E$-polynomials, describe the topological mirror symmetry
conjecture of Hausel--Thaddeus, and state and prove our result. Section~\ref{sec:(11...1)} is devoted
to the calculation of the contribution to the variant part of the $E$-polynomial of the
$\SL(n,\C)$-moduli space arising only from the $\C^*$-fixed point loci of type $(1,1,\ldots,1)$. In
Section~\ref{sec:unramified-Norm,Prym,actions} we recall some
classical results on Prym varieties of unramified covers. These are
used in Section~\ref{sec:stringy E-pol}, where the contribution from
the fixed point loci of non-trivial elements of $\Gamma_n$ to the stringy
$E$-polynomial of $\PGL(n,\C)$-moduli space is calculated.

\subsection*{Acknowledgments}

We thank David Alfaya, Emilio Franco, Laura Fredrickson, Oscar García-Prada, Tomas Gómez,
Tamás Hausel, Jochen Heinloth and Ana Pe\'on-Nieto for useful discussions.


\section{Parabolic Higgs bundles and their moduli}
\label{sec:parabolic-higgs-bundles-and-moduli}

In this section we recall basic facts about parabolic Higgs bundles
and their moduli spaces.

\subsection{Parabolic vector bundles}

Denote by $X$ a smooth projective curve over $\C$, and mark it with
distinct points labeled by the divisor \[D=p_1+\cdots+p_{|D|},\]
with $p_i\neq p_j$ for $i\neq j$ and where $|D|=\deg{D}$. Let $g$ be
the genus of $X$ and assume $g\geq 2$.  This data will be fixed
throughout.

Parabolic vector bundles on $X$ associated to $D$, are vector bundles
together with extra structure over each point of $D$.

\begin{definition}\label{def:parabvb}
  A holomorphic \emph{parabolic vector bundle of rank $n$ on $X$,
    associated to the divisor $D$}, is a holomorphic vector bundle $V$
  of rank $n$ over $X$, endowed with a \emph{parabolic structure along
    $D$}. By this is meant a collection of weighted flags of the
  fibers of $V$ over each point $p\in D$:
\begin{equation}\label{eq:parabstructure1}
\begin{split}
V_p=V_{p,1} \supsetneq V_{p,2} \supsetneq & \cdots \supsetneq V_{p,s_p}\supsetneq V_{p,s_p+1}=\{0\},\\
0\leq \a_1(p) < & \dots <\a_{s_p}(p) < 1,
\end{split}
\end{equation}
where $s_p$ is an integer between $1$ and $n$. 
The real number $\alpha_i(p)\in[0,1)$ is \emph{the weight of the
  subspace $V_{p,i}$}. The \emph{multiplicity of the weight $\a_i(p)$}
is the number $m_i(p)=\dim(V_{p,i}/V_{p,i+1})$, thus
$\sum_im_i(p)=n$.  The data given only by the flags over $D$ (i.e.,
without the weights) is called the \emph{quasi-parabolic structure} of
$V$. The parabolic structure is obtained from a quasi-parabolic
structure by specifying the weights.
The \emph{type of the quasi-parabolic structure} is $\bm m$, where
$\bm m=(m_1(p),\ldots,m_{s_p}(p))_{p\in D}$ is the collection of all
multiplicities over all points of $D$.  The \emph{type of the
  parabolic structure} is $(\bm m,\bm\a)$, with
$\bm\a=(\a_1(p),\ldots,a_{s_p}(p))_{p\in D}$ being the collection of
all weights.  The \emph{type} of the parabolic vector bundle is
$(n,d,\bm m,\bm\a)$, where $d=\deg(V)$ is its degree.
Finally, a flag over a point $p\in D$ is \emph{full} if $s_p=n$ or,
equivalently, $m_i(p)=1$ for all $i$.
\end{definition}

We shall denote a parabolic vector bundle by just $V$ whenever the
parabolic structure is clear from the context.

 
\begin{remark}\label{rmk:parabtensorusual}
  Given a parabolic vector bundle $V$, with parabolic structure of
  type $(\bm m,\bm\a)$, and a line bundle $L$ the tensor product
  $V\otimes L$ acquires a parabolic structure, of the same type
  $(\bm m,\bm\a)$, in the obvious way, i.e., by taking the flags on
  $V\otimes L$ along $D$ induced by the ones of $V$, with the same
  weights. Except when explicitly mentioned to the contrary, this will
  be the parabolic structure we shall consider on $V\otimes L$.
  In fact, it corresponds to the general tensor product of parabolic
  bundles (see Yokogawa \cite{yokogawa:1995}) in the particular case
  where $L$ has trivial parabolic structure.
\end{remark}
 
Next we come to morphisms of parabolic bundles. These will be the
vector bundle homomorphisms which preserve the parabolic
structures; however these can be preserved in a week or a strong sense.

\begin{definition}
  Let $V$ and $W$ be parabolic vector bundles whose parabolic
  structures are of type $(\bm m,\bm\a)$ and $(\bm l,\bm\beta)$
  respectively, and let $\p: V\to W$ be a holomorphic map.  The map $\p$
  is called \emph{parabolic} if we have, for all $p\in D$,
\[\a_i(p) >\beta_j(p)\Longrightarrow \p(V_{p,i}) \subseteq W_{p,j+1}.\]
Denote by $\PH(V,W)$ the bundle of parabolic homomorphisms from $V$ to
$W$ and, if $W=V$, write $\PE(V)$ instead.  The map $\p$ is said
\emph{strongly parabolic} if
\[\a_i(p)\geq \beta_j(p)\Longrightarrow \p(V_{p,i}) \subseteq W_{p,j+1},\]
for all $p\in D$. Denote by $\SPH(V,W)$ the bundle of strongly
parabolic homomorphisms from $V$ to $W$ and, if $W=V$, write $\SPE(V)$
instead.
\end{definition}

\subsection{Parabolic Higgs bundles}
We shall need to consider parabolic Higgs bundles with various
structure groups $G$. Indeed, there is a theory of parabolic $G$-Higgs
bundles (see, for example, Biquard--García-Prada--Mundet
\cite{biquard-garcia-prada-mundet:2015} for a general notion for real
reductive $G$) but, since we shall only need the groups
$\GL(n,\C)$, $\SL(n,\C)$ and $\PGL(n,\C)$, we can make the
following ad hoc definitions.

Let $K=\Omega_X^1$ be the canonical bundle on $X$ and write
$K(D)=K\otimes\cO_X(D)$.

\begin{definition}
  A {\em strongly parabolic $\GL(n,\C)$-Higgs bundle} is a pair
  $(V,\vp)$, where $V$ is a parabolic bundle of rank $n$ and the
  \emph{Higgs field} $\vp:V\to V\otimes K(D)$ is a strongly parabolic
  homomorphism, i.e., $\vp$ is a holomorphic section of $\SPE(V)\otimes K(D)$,
  where $V\otimes K(D)$ has the parabolic structure defined by $V$
  (cf.\ Remark \ref{rmk:parabtensorusual}).
The \emph{type} of a parabolic $\GL(n,\C)$-Higgs bundle $(V,\vp)$ is
the type of the parabolic vector bundle $V$; cf. Definition
\ref{def:parabvb}.
\end{definition}

Thus, in a strongly parabolic Higgs bundle $(V,\vp)$, the Higgs field
$\vp$ is a meromorphic endomorphism valued one-form with at most
simple poles along $p\in D$ and whose residue at $p$ is nilpotent with
respect to the flag. In other words, if the parabolic structure on $V$
is given by (\ref{eq:parabstructure1}) then,
\begin{displaymath}
  \vp(V_{p,i})\subseteq V_{p,i+1}\otimes K(D)_p.
\end{displaymath}

\begin{remark}
  If we require $\vp$ to be just parabolic, rather than strongly
  parabolic, we get the notion of parabolic Higgs bundle (for the
  structure groups considered). We shall, however, never use this notion
  in the present paper. Thus we shall frequently omit the adverb
  ``strongly'', but the reader should keep in mind that the Higgs
  field $\vp$ is always required to be strongly parabolic.
\end{remark}

 
If $V$ is a parabolic bundle of rank $n$, consider the determinant
line bundle $\Lambda^nV$. Though this has a natural parabolic
structure, in the following definition we ignore it and consider
just the underlying line bundle.
 
\begin{definition}
  Fix a holomorphic line bundle $\Lambda$ on $X$ of degree $d\in\Z$. A
  {\em strongly parabolic $\SL(n,\C)$-Higgs bundle with fixed
    determinant $\Lambda$} is a pair $(V,\vp)$, where $V$ is a
  parabolic bundle of rank $n$ such that $\Lambda^nV\cong\Lambda$, and
  where $\vp\in H^0(X,\SPE(V)\otimes K(D))$ is such that
  $\tr(\vp)\equiv 0$.
\end{definition}

Note that, strictly speaking, ``$\SL(n,\C)$-bundle'' should only refer
to the case where the line bundle $\Lambda$ is trivial, so we are
committing a slight abuse of language here. 

When there is no need to specify the structure group, or when it is
clear from the context, we shall often make a further innocuous abuse
of language and say simply (strongly) parabolic Higgs bundle.

In order to introduce strongly parabolic $\PGL(n,\C)$-Higgs bundles,
recall that any holomorphic $\PGL(n,\C)$-bundle over the curve $X$
lifts to a holomorphic vector bundle $V\to X$, and that two such lifts
$V$ and $V'$ differ by tensoring by a line bundle.

\begin{definition}
  \label{def:PGL-PHB}
  A {\em strongly parabolic $\PGL(n,\C)$-Higgs bundle} is an
  equivalence class $[(V,\vp)]$ of strongly parabolic
  $\GL(n,\C)$-Higgs bundles, where $(V,\vp)$ and $(V',\vp')$ are
  considered equivalent if there is a line bundle $L$ such that
  $V'\cong V\otimes L$, the parabolic structure of $V'$ is the one
  obtained from $V$, and $\vp'=\vp\otimes\Id_L$.
\end{definition}

\begin{remark}
  \label{rem:PGL-fixed-determinant}
  Recall that $\PGL(n,\C)$-bundles over the curve $X$ are topologically classified by
  $\pi_1(\PGL(n,\C)) \cong \Z_n$.
  Fixing a topological type
  $c\in\Z_n$ and a holomorphic line bundle $\Lambda$ whose degree
  modulo $n$ equals $c$, any holomorphic $\PGL(n,\C)$-bundle of
  topological type $c$ may be lifted to a holomorphic vector bundle
  whose determinant bundle is isomorphic to $\Lambda$.  Moreover, two
  lifts with the same determinant bundle differ by tensoring by a line
  bundle which is a $n$-torsion point of the Jacobian $\Jac(X)$. These
  facts reflect the identifications
  \(\PGL(n,\C)=\GL(n,\C)/\C^*=\SL(n,\C)/\Z_n=\PSL(n,\C).\)
\end{remark}

\subsection{Stability and moduli spaces}
In the following we recall the stability condition for parabolic Higgs
bundles and introduce their moduli spaces.  

\begin{definition}
  Given a parabolic vector bundle $V$, a \emph{parabolic subbundle} is
  a vector subbundle $V'\subseteq V$, with the parabolic structure
  defined as follows. For each $p\in D$, the quasi-parabolic structure
  is given by the flag
  \[V'_p=V'_{p,1} \supsetneq V'_{p,2} \supsetneq \cdots \supsetneq
  V'_{p,s'_p}\supsetneq \{0\},\]
  where $V'_{p,i}=V'_p\cap V_{p,i}$, discarding all the repetitions of
  subspaces in the filtration.
Moreover, the weights $0\leq \a'_1(p)< \cdots <\a'_{s'_p}(p) < 1$ are
taken to be the greatest possible among the corresponding original
weights, meaning that
\begin{equation}\label{eq:weight-subbundle}
  \a'_i(p)=\max_j\{\a_j(p)\st V_{p,j}\cap V'_p=V'_{p,i}\}
  =\max_j\{\a_j(p)\st V'_{p,i}\subseteq V_{p,j}\}.
\end{equation}
In other words, the weight attached to $V'_{p,i}$ is the weight
$\a_j(p)$ whose index $j$ is such that $V'_{p,i}\subseteq V_{p,j}$ but
$V'_{p,i}\not\subseteq V_{p,j+1}$.
\end{definition}

\begin{definition}
  The \emph{degree} of a parabolic Higgs bundle $(V,\vp)$
  is the degree of the underlying bundle, $\deg(V)\in\Z$.  The
  \emph{parabolic degree} $\pdeg(V)$ and \emph{parabolic slope}
  $\pmu(V)$ of $(V,\vp)$ are the parabolic degree and slope,
  respectively, of the underlying parabolic vector bundle, defined by
  \[\pdeg(V)=\deg(V)+\sum_{p\in D} \sum_{i=1}^{s_p}m_i(p)\a_i(p)\qquad
  \text{and}\qquad \pmu (V)=\frac{\pdeg (V)}{n}.\]
\end{definition}

\begin{definition}\label{def:stabcond}
  A strongly parabolic Higgs bundle $(V,\vp)$ is \emph{semistable} if
  \[\pmu(V')\leq\pmu(V)\]
  for every non-zero parabolic subbundle $V'\subseteq V$ which is
  \emph{$\vp$-invariant}, that is, $\vp(V')\subseteq V'\otimes K(D)$.
  It is \emph{stable} if it is semistable and strict inequality holds
  above for all proper non-zero $\vp$-invariant parabolic subbundles
  $V'\subseteq V$.
\end{definition}

Consider now quasi parabolic Higgs bundles of rank $n$, degree $d$,
and quasi-parabolic type $\bm m$.  The space of compatible parabolic
weights $\bm{\a}=(\a_1(p),\ldots,\a_{s_p}(p))_{p\in D}$ is a product
$\mathcal{S}$ of simplices (excluding some boundaries) determined by
the inequalities in (\ref{eq:parabstructure1}), one simplex for each
point of $D$.  Let $(V,\vp)$ be a parabolic Higgs bundle of type
$(\bm{m},\bm\alpha)$. If $(V,\vp)$ is semistable but not stable, then
\begin{equation}\label{eq:non-genweights}
  n\bigg(d'+\sum_{p\in D} \sum_{i=1}^{s'_p} m'_i(p)\a_i'(p)\bigg)=n'\bigg(d+\sum_{p\in D} \sum_{i=1}^{s_p}m_i(p)\a_i(p)\bigg),
\end{equation} 
where $d'$ and $n'$ are the degree and rank of a destabilizing
parabolic Higgs subbundle $(V',\vp|_{V'})$ with multiplicities
$m'_i(p)$ and weights $\a'_i(p)$. For given $d'$, $n'$ and $m'_i(p)$, equation \eqref{eq:non-genweights}
determines an intersection of a hyperplane with $\mathcal{S}$ which is
called a \emph{wall}.  There are only finitely many possible values for
$n'$ and $m'_i(p)$ and, for each of these, there are also finitely
many values of $d'$ for which these hyperplanes intersect
$\mathcal{S}$. Hence there are finitely many walls.

\begin{definition}
  Fix $n$, $d$ and a quasi-parabolic type $\bm m$. A weight vector
  $\bm\alpha\in\mathcal{S}$ is called \emph{generic} if it does not
  belong to a wall. A connected component of the complement of the set
  of walls is called a \emph{chamber}.
\end{definition}

\begin{remark}
  It is immediate from this definition that for generic weights, a
  semistable parabolic Higgs bundle is in fact stable. Moreover, for
  generic weights in the same chamber, the stability condition
  is unchanged, so the corresponding moduli spaces (to be
  introduced presently) will be isomorphic.
\end{remark}

A GIT construction of the moduli space $\cM_d^{\bm m,\bm\a}(\GL(n,\C))$ of semistable parabolic
$\GL(n,\C)$-Higgs bundles over $X$, of rank $n$, degree $d$ and parabolic type
$(\bm{m},\bm\a)$, was carried out by Yokogawa \cite{yokogawa:1993}, and
the deformation theory of parabolic Higgs bundles was also studied by
Yokogawa \cite{yokogawa:1995} (cf.\ Boden--Yokogawa
\cite{boden-yokogawa:1996}). A gauge theoretic construction of the
moduli space of (non-strongly) parabolic Higgs bundles was done by
Konno \cite{konno:1993}. It was proved by Yokogawa that the stable
locus of the moduli space is smooth and quasi-projective. Thus we have
the following result.

\begin{proposition}
  Assume that the weights $\bm\alpha$ are generic. Then the moduli
  space 
  $\cM_d^{\bm m,\bm\a}(\GL(n,\C))$ is a smooth quasi-projective
  variety. Moreover, for generic weights in the same chamber of
  $\mathcal{S}$, the corresponding moduli spaces are isomorphic. \qed
\end{proposition}

In order to obtain the moduli space of parabolic $\SL(n,\C)$-Higgs
bundles over $X$, consider the determinant map
\begin{align*}
  p:\cM_d^{\bm m,\bm\a}(\GL(n,\C))
    &\to T^*\Pic^d(X)\cong\Pic^d(X)\times H^0(X,K),\\
  (V,\vp)&\mapsto (\Lambda^nV,\tr(\vp)),
\end{align*}
with $\Pic^d(X)$ the component of the Picard variety of $X$ of degree
$d$ line bundles.  Notice that, since $(V,\vp)$ is strongly parabolic,
the residue of the trace $\tr(\vp)$ vanishes along $D$, and so it is
in fact a section of $K$.  
Let
\[
\cM_\Lambda^{\bm m,\bm\a}(\SL(n,\C))=p^{-1}(\Lambda,0).
\]
For generic weights, this is again a smooth quasi-projective variety.
Since this is the moduli space we shall mostly be working with, whenever
there is no risk of confusion, we shall denote it simply by $\cM$.

Next we want to introduce the moduli space of parabolic
$\PGL(n,\C)$-Higgs bundles. In view of Definition~\ref{def:PGL-PHB}
and Remark~\ref{rem:PGL-fixed-determinant} we consider the group 
\[\Gn=\Jac_{[n]}(X)=\{L\in\Jac(X)\st L^n\cong\mathcal
O_X\}\subset\Jac(X)\]
of $n$-torsion points of the Jacobian of $X$. Recall that
$\Gn\cong H^1(X,\Z_n)\cong\Z_n^{2g}$.  It will be convenient to
distinguish the elements of $\Gn$ as an abstract group and as line
bundles; thus, if $\gamma$ denotes an element of $\Gn$, the
corresponding line bundle will be denoted by $L_\gamma$. Fix a line
bundle $\Lambda$
and let $[d]\in\Z_n$ denote the reduction of $d=\deg(\Lambda)$ modulo
$n$. The group
$\Gamma_n$ acts on $\mathcal{M}$ by
\begin{equation}\label{eq:Gamma_n action}
\gamma\cdot (V,\vp)=(V\otimes L_\gamma,\vp\otimes\Id_{L_\gamma})
\end{equation}
(note that $\Gn$ acts trivially on the parabolic structure).
We take the moduli space of
parabolic $\PGL(n,\C)$-Higgs bundles of topological type $[d]$ to be 
\[\cM_{[d]}^{\bm m,\bm\a}(\PGL(n,\C))=\cM/\Gamma_n.\] 
We remark that this is consistent with the abstract definition of
stability of parabolic $G$-Higgs bundles coming from
\cite{biquard-garcia-prada-mundet:2015}.  As opposed to $\cM$, the
moduli space $\cM/\Gn$ is not smooth, but rather an orbifold, with
singularities arising from the fixed points of the action of
$\Gn$. 

Serre duality for a parabolic vector bundles (see
\cite{yokogawa:1995,boden-yokogawa:1996}) says that for a parabolic
vector bundle $V$ 
\begin{displaymath}
  H^1(\PE(V)) \simeq H^0(\SPE(V)\otimes K(D))^*
\end{displaymath}
(and analogously in the traceless case), in other words, the
infinitesimal deformation space of $V$ is dual to the space of
Higgs fields on $V$. Thus, letting $\cN$ denote the moduli space of
parabolic vector bundles (with fixed determinant $\Lambda$), there is
an embedding of the cotangent bundle $T^*\cN\into\cM$ as an open
subset. The natural symplectic structure on the cotangent bundle
extends to $\cM$, which is thus a holomorphic symplectic manifold (see
Bottacin \cite[Sec.~5]{bottacin:1995}, Biswas--Ramanan
\cite[Sec.~6]{biswas-ramanan:1994}, Konno \cite{konno:1993}, 
Logares--Martens \cite{logares-martens:2010}, and cf.\ Yokogawa
\cite{yokogawa:1995} and Boden--Yokogawa
\cite{boden-yokogawa:1996}). Moreover, Konno's gauge theoretic
construction (loc.\ cit.) shows that $\mathcal{M}$ has a compatible
hyper-K\"ahler metric.




\subsection{Isomorphism between moduli spaces for different degrees and weights}

Let $\Lambda$ and $\Lambda'$ be line bundles on $X$, not necessarily
of the same degree. In this section we prove that, under mild
conditions on the parabolic structure $\bm\a$, one can find a
parabolic structure $\bm\a'$ so that the moduli spaces
$\cM_\Lambda^{\bm\a}(\SL(n,\C))$ and
$\cM_{\Lambda'}^{\bm\a'}(\SL(n,\C))$ are isomorphic. This expands on
\cite[Proposition~2.1]{garcia-prada-gothen-munoz:2007} and marks a
substantial difference to the non-parabolic case, where such an
isomorphism can only exist if $\deg(\Lambda)$ and $\deg(\Lambda')$ are
equal modulo $n$.

We need the notion of tensor product of parabolic (Higgs)
bundles. This is better viewed in the more general context of
parabolic (or filtered) sheaves (see Boden--Yokogawa
\cite{boden-yokogawa:1996}, Yokogawa \cite{yokogawa:1995}, and Simpson
\cite{simpson:1990}) but we shall only need a few simple facts which
we now review. In fact it suffices for us to consider the case when
one of the bundles is a parabolic line bundle (with trivial Higgs
field), so let $(V,\bm\alpha)$ be a parabolic vector bundle bundle of
rank $n$ and let $(L,\bm\beta)$ be a parabolic line bundle on $X$. There
is a \emph{parabolic tensor product} of the parabolic bundles $V$ and
$L$, denoted by $(V\otimes ^P L,\bm\alpha')$. The parabolic weights
$\bm\alpha'$ are given by
\begin{equation}
  \label{eq:p-tensor-weights}
  \alpha_i'(p) =
  \begin{cases}
    \alpha_i(p) + \beta(p) &\text{if $\alpha_i(p) + \beta(p)<1$},\\
    \alpha_i(p) + \beta(p)-1 &\text{if $\alpha_i(p) + \beta(p)\geq1$},
  \end{cases}
\end{equation}
where for each $p\in D$ the correct ordering of the $ \alpha_i'(p)$ by size
corresponds to a cyclic permutation of the ordering of the
indices $i=1,\dots,s_p$. The multiplicity of the weight $\alpha_i'(p)$ is
$m_i(p)$. The parabolic degree of the parabolic tensor product is
given by the usual formula:
\begin{align*}
  \pdeg(V\otimes^P L) &= \pdeg(V) + n\pdeg(L)\\
  &= \deg(V)+\sum_{p,i}m_i(p)\alpha_i(p)+n\Big(\deg(L)+\sum_p\beta(p)\Big).
\end{align*}
In view of this we get the following formula for the (non-parabolic) degree of
$V\otimes ^P L$:
\begin{equation}
\label{eq:usual-deg-VL}
\begin{aligned}
  \deg(V\otimes^P L) &= \pdeg(V\otimes^P L) - \sum_{p,i}m_i(p)\alpha_i'(p)\\
  &= \deg(V) + n\deg(L)
    + \sum_{p,i}m_i(p)(\alpha_i(p)+\beta(p)-\alpha_i'(p)).
\end{aligned}
\end{equation}
Finally we remark that if $V$ underlies a parabolic Higgs bundle
$(V,\vp)$, then $V\otimes^P L$ has a Higgs field induced by
$\vp\otimes\Id_L$ and that $(V,\vp)$ is stable if and only if $(V\otimes^P
L,\vp\otimes\Id_L)$ is (see Simpson \cite{simpson:1990}).
\begin{theorem}
  Let $\bm\alpha=(\alpha_i(p))_{p\in D}$ be a system of parabolic
  weights and let $\Lambda$ and $\Lambda'$ be line bundles of degrees
  $d$ and $d'$, respectively.  Suppose there are 
  $\bm\beta=(\beta(p))_{p\in D}$ in $[0,1)$ such that
\begin{equation}\label{eq:cond-isomorphism-parabolicmodulispaces}
d'-d-\sum_{p,i}m_i(p)(\alpha_i(p)+\beta(p)-\alpha_i'(p))\equiv 0\ \pmod{n}.
\end{equation} 
Then there is a parabolic line bundle $(L,\bm\beta)$ inducing an
isomorphism
\begin{align*}
  \cM_\Lambda^{\bm\a}(\SL(n,\C)) &\xra{\cong}
                                   \cM_{\Lambda'}^{\bm\a'}(\SL(n,\C)),\\
  (V,\vp) &\mapsto (V\otimes^P L,\vp\otimes\Id_L),
\end{align*}
where the weights $\bm\alpha'$ are given by \eqref{eq:p-tensor-weights}.
\end{theorem}

\proof
In view of \eqref{eq:usual-deg-VL} and the remarks preceding the
statement of the theorem, we can find a line bundle $L$ such that
tensoring by $L$ gives an isomorphism 
$\cM_d^{\bm m,\bm\a}(\GL(n,\C))\xra{\cong}\cM_{d'}^{\bm m,\bm\a'}(\GL(n,\C))$. In
order to get the isomorphism between the fixed determinant moduli
spaces it suffices to adjust $L$ by tensoring by a suitable (non-parabolic) degree
zero line bundle.
\endproof

The following corollary describes two situations where it is possible to find $\bm\beta$ satisfying \eqref{eq:cond-isomorphism-parabolicmodulispaces}. The conditions imposed are mild, and since we will at some point impose full flags at all points of $D$, we will be under the assumptions of this corollary.

\begin{corollary}\label{cor:isomdifferentdegreesandweights-fullflag}
Consider the moduli space $\cM_\Lambda^{\bm m,\bm\a}(\SL(n,\C))$.
\begin{enumerate}
\item If $\Lambda'$ is any line bundle of degree $d'$ such that $d'\equiv d\ (\mathrm{mod}\, n)$, then $\cM_\Lambda^{\bm m,\bm\a}(\SL(n,\C))\cong\cM_{\Lambda'}^{\bm m,\bm\a}(\SL(n,\C))$.
\item Suppose the parabolic structure $\bm\a$ is such that for some
  point $p_0\in D$ the flag is full (i.e., $s_{p_0}=n$). Then, given
  any line bundle $\Lambda'$ of any degree $d'$, there exists a
  parabolic structure $\bm\a'$ such that
  $\cM_{\Lambda'}^{\bm m,\bm\a'}(\SL(n,\C))\cong\cM_\Lambda^{\bm m,\bm\a}(\SL(n,\C))$. 
\end{enumerate}
\end{corollary}
\proof For the first item, we just have to take $L$ a $n$th root of
$\Lambda'\Lambda^{-1}$ and use the usual tensor product
$V\mapsto V\otimes L$. This is of course the generalization to the
parabolic case of the classical isomorphism in the non-parabolic case.

For the second item, suppose that $d'-d\equiv k\ (\mathrm{mod}\,
n)$. Since the flag over $p_0$ is full, we can choose $\beta(p_0)\in[0,1)$
such that
$\#\{i \st \a_i(p_0)+\beta(p_0)\geq 1\} = k$.
If $p\in D\setminus\{p_0\}$, take $\beta(p)=0$, so 
$\#\{i \st \a_i(p)+\beta(p)\geq 1\} = 0$.
With these choices, \eqref{eq:cond-isomorphism-parabolicmodulispaces} holds in view of \eqref{eq:p-tensor-weights}, and the conclusion follows by the theorem.
\endproof


%

\subsection{Basic assumptions}

We now make two assumptions. 
\begin{assumption}\label{assump:generweights}
We assume from now on that:
\begin{enumerate}
\item the weights $\bm\a$ are generic;
\item the flags over all points of $D$ are full (i.e., $m_i(p)=1$ for all $i$ and all $p\in D$, thus $s_p=n$ for all $p$).
\end{enumerate}
\end{assumption}

Since from now on $m_i(p)=1$ for all $i,p$, we shall remove the
$\bm m$ from the notation.

The first assumption is essential for us. It implies that any
semistable parabolic Higgs bundle is stable and hence, as shown by
Yokogawa \cite{yokogawa:1993}, the moduli space is smooth.

The second assumption serves two purposes. Firstly, the SYZ mirror
symmetry picture (outlined in the next section) has currently only
been shown under this assumption. Secondly, it simplifies the formulas
in our calculations of Hodge polynomials. We point out, however, that
these calculations generalize without too much trouble to the
case of general flags.

Summarizing, under Assumption~\ref{assump:generweights}, the moduli
space of parabolic Higgs bundles $\mathcal{M}$ is a smooth
quasi-projective hyper-K\"ahler manifold. Its dimension can be
calculated using deformation theory (see, for example,
\cite[Proposition 2.4]{garcia-prada-gothen-munoz:2007}) and is given by
\begin{equation}\label{eq:dim}
\dim(\cM)=2(n^2-1)(g-1)+|D|n(n-1),
\end{equation}
where we recall that $|D|=\deg(D)$ is the number of marked points on $X$.


\section{Mirror symmetry}
\label{sec:mirror-symmetry}

In this section we recall the Hausel--Thaddeus mirror symmetry
proposal in the parabolic case. First, in Section~\ref{subsec:Hitmap}
we treat the Hitchin system and mirror symmetry according to
Strominger--Yau--Zaslow. Next, in
Section~\ref{sec:stringy-E-polynomial}, we recall the definition of
the stringy $E$-polynomial and show, following Thaddeus, its
independence of the parabolic weights. Finally, in
Section~\ref{sec:topological-MS} we state our main result and outline
its proof.

\subsection{The Hitchin map and SYZ mirror symmetry}
\label{subsec:Hitmap}

In this section we briefly describe how $\cM$ and $\cM/\Gn$ are mirror
partners in the sense of Strominger--Yau--Zaslow (SYZ)
\cite{strominger-yau-zaslow:1993}.  This has been shown in the
parabolic case by Biswas--Dey \cite{biswas-dey:2012} (following
Hausel--Thaddeus \cite{hausel-thaddeus:2003}). The general version of
SYZ mirror symmetry proved by these authors requires considering a
naturally defined gerbe (or $B$-field) on the moduli spaces (see also
Donagi--Pantev \cite{donagi-pantev:2012}). As explained below, the
statement of SYZ mirror symmetry involves identifying fibers of the
Hitchin maps of the two moduli spaces as dual abelian varieties. The
need for introducing the $B$-field comes from the lack of a natural
base point in these fibers. In the parabolic case there is a twist in
the story: the moduli spaces $\cM$ and $\cM/\Gn$ are also
mirror partners in the original sense of SYZ.

\begin{remark}
\label{rem:SYZ-de-Rham-Dolbeault}
  We recall that the true mirror partners are in fact the \emph{de
    Rham moduli spaces}; these are moduli spaces of local systems on
  $X$ and are diffeomorphic to the Higgs bundle moduli spaces under
  the non-abelian Hodge correspondence. As explained in
  \cite[Sec.~1]{hausel-thaddeus:2003} in the non-parabolic case, the
  statements on the de Rham side can be translated into statements on
  the Higgs bundle side through a hyper-K\"ahler rotation, and this
  works exactly the same way in the parabolic case. We refer the
  reader to Simpson \cite{simpson:1990} and Alfaya--Gómez
  \cite{alfaya-gomez:2016} for details on the de Rham moduli spaces in
  the parabolic case.
\end{remark}

We now introduce the Hitchin system in the parabolic setting.  This
goes back
to Bottacin \cite{bottacin:1995} and Nasatyr--Steer
\cite{nasatyr-steer:1995}. We start by defining the \emph{Hitchin
map} $h$ on the moduli spaces $\cM$ and $\cM/\Gn$: it
takes a parabolic Higgs bundle $(V,\vp)$ to the coefficients of the
characteristic polynomial of the twisted endomorphism
$\vp:V\to V\otimes K(D)$. Thus $h(V,\vp)=(s_2,\ldots,s_n)$, with
$s_i=\tr(\wedge^i\vp)$.  Since $\vp$ is strongly parabolic, its
restriction to every $p\in D$ is nilpotent, and so all the corresponding
coefficients $s_i(p)$ of the characteristic polynomial vanish. We
therefore have
\begin{equation}\label{eq:Hitchinmap}
  \begin{aligned}
    h\colon\mathcal{M}&\to \mathcal{A}=\bigoplus_{i=2}^nH^0(X,K^i((i-1)D)),\\
    (V,\vp)&\mapsto (s_2,\ldots,s_n),
  \end{aligned}
\end{equation}
where $\mathcal{A}$ is the
\emph{Hitchin base}.
It is clear that $h$ factors through the quotient $\cM/\Gn$, so we
also have a Hitchin map $h'$ on this moduli space:
\[
\xymatrix{\cM\ar[dr]_{h}&&\cM/\Gn\ar[dl]^{h'}\\&\mathcal{A}&}
\]  
Observe that
\[
\dim(\mathcal{A})=(n^2-1)(g-1)+\frac{n(n-1)|D|}{2}=\frac{\dim(\cM)}{2}.
\]

By \cite{yokogawa:1993}, the map $h$ is proper, hence so is
$h'$. The coordinate functions of $h$ and $h'$ are independent and Poisson commute, and
these maps form the \emph{Hitchin systems} for $\SL(n,\C)$ and
$\PGL(n,\C)$, respectively. In particular, for $s\in\mathcal{A}$, the fibers $h^{-1}(s)$ and $h^{\prime -1}(s)$ are complex Lagrangian subvarieties of $\cM$ and $\cM/\Gn$.

To describe the generic Hitchin fibers more precisely, consider the
quasi-projective surface given by the total space $|K(D)|$ of $K(D)$
and the  projection $\pi\colon|K(D)|\to X$. Given a point
$s=(s_2,\ldots,s_n)$ in the Hitchin base $\mathcal{A}$, there is a
projective curve $X_s$, lying in $|K(D)|$, defined by the zeros of the
section
\begin{equation}\label{eq:section-spectral curve}
\lambda^n+\pi^*s_2\lambda^{n-2}+\cdots+\pi^*s_n\in H^0(|K(D)|,\pi^*(K^n(nD)))
\end{equation}
where $\lambda$ is the tautological section of $\pi^*(K(D))$ and
$\lambda^{n-i}\pi^*s_i\in H^0(|K(D)|,\pi^*(K^n((n-1)D)))\subseteq
H^0(|K(D)|,\pi^*(K^n(nD)))$. The curve $X_s$ is called the
\emph{spectral curve} associated to $s\in\mathcal{A}$.  The
restriction of $\pi$ to $X_s$ gives an $n$-cover $\pi:X_s\to X$ which
is ramified over the locus where \eqref{eq:section-spectral curve}
has multiple roots. This locus is always non-empty.

By Lemma 3.1 of \cite{gomez-logares:2011}, there is an open and dense
subspace $\mathcal{A}'\subset  \mathcal{A}$ such that $X_s$ is smooth
whenever $s\in \mathcal{A}'$ (this needs the assumption $g\geq 2$ on the genus of $X$). Moreover, for such generic $s$, Lemma
3.2 of \cite{gomez-logares:2011} states that the fibre $h^{-1}(s)$ is
naturally isomorphic to 
\begin{equation}
  \label{eq:fibre-Pd}
P^{d'}=\{L\in\Pic^{d'}(X_s)\st\det(\pi_*L)\cong\Lambda\},
\end{equation}
where $d'=d+n(n-1)(g-1+|D|/2)$.

\begin{remark}
  Lemma 3.2 of \cite{gomez-logares:2011} needs the full flags
  assumption on every point of the divisor $D$. This is one reason why
  we confine ourselves to the full flag condition.
\end{remark}

\begin{definition}\label{def:NmDiv}
Consider a degree $n$ cover $\pi:Y\to X$. The \emph{norm map} between the groups of divisors $\Nm_\pi\colon\Div(Y)\to\Div(X)$ is the homomorphism taking a divisor $E=\sum E(p)p$ on $Y$ to the divisor $\Nm_\pi(E)=\sum E(p)\pi(p)$ on $X$.
\end{definition}

The norm map just defined factors through the norm map between the Picard groups $\Nm_\pi\colon\Pic^0(X_s)\to\Pic^0(X)$, by $\Nm_\pi(L)=L'$ where $L\cong\cO_Y(E)$ and $L'\cong\cO_X(\Nm_\pi(E))$. 

\begin{definition}\label{def:Prym}
The \emph{Prym variety of $Y$} associated to $\pi$, denoted by
$\Prym_\pi(Y)$, is the abelian subvariety of $\Pic^0(Y)$ defined as the identity component of the kernel of $\Nm_\pi$. 
\end{definition}

The kernel of $\Nm_\pi$ is connected if and only if $\pi$ is
ramified and
sometimes the term Prym variety is used for the full kernel of $\Nm_\pi$. We have adopted Definition~\ref{def:Prym} in accordance with \cite{hausel-thaddeus:2003}.  
Since for  $s\in\mathcal A'$ the cover $\pi:X_s\to X$ is ramified\footnote{In Section \ref{sec:unramified-Norm,Prym,actions} we shall need to consider the norm map and corresponding Prym of unramified covers.}, we have
\begin{equation}\label{eq:Prym=ker-ramified}
\Prym_{\pi}(X_s)=\ker(\Nm_\pi)=\{L\in\Pic^0(X_s)\st\Nm_\pi(L)\cong\cO_X\}.
\end{equation}
Note
that 
\begin{equation}\label{eq:det-norm}
\det(\pi_*L)\cong\Nm_\pi(L)\otimes\det(\pi_*(\cO_{X_s}))\cong\Nm_\pi(L)\otimes
(K(D))^{-n(n-1)/2}.
\end{equation}
Thus we see that
\[P^{d'}\cong\{L\in\Pic^{d'}(X_s)\st\Nm_\pi(L)\cong\Lambda (K(D))^{n(n-1)/2}\}\]
is a torsor for $\Prym_\pi(X_s)$. 

It is also easy to see that $h^{\prime -1}(s)$ is isomorphic $P^{d'}/\Gn$, hence it is a torsor for the quotient $\Prym_\pi(X_s)/\Gn$, where $\Gn$ acts by tensoring by the pull-back via $\pi$. The quotient $\Prym_\pi(X_s)/\Gn$ is an abelian variety, isogenous to $\Prym_\pi(X_s)$.

By dualising the short exact sequence coming from the norm map, one easily checks \cite[Lemma 2.3]{hausel-thaddeus:2003} that these two abelian varieties are dual to each other, in the sense that 
\[\Pic^0(\Prym_\pi(X_s))\cong\Prym_\pi(X_s)/\Gn\qquad \text{ and }\qquad \Pic^0(\Prym_\pi(X_s)/\Gn)\cong\Prym_\pi(X_s).\]

\begin{theorem}[Hausel--Thaddeus \cite{hausel-thaddeus:2001}]\label{thm:SYZmirror}
Assume that $s\in\mathcal{A}$ has simple zeros. Then the Hitchin
fibers $h^{-1}(s)$ and $h^{\prime -1}(s)$ can be naturally identified with
a pair of dual abelian varieties. Hence $\mathcal{M}$ and
$\mathcal{M}/\Gamma_n$ are SYZ mirror partners.
\end{theorem}
\proof Assume first that $\Lambda\cong K(D)^{-n(n-1)/2}$ (so that
$d=-n(n-1)(g-1+|D|/2)$ and $d'=0$). In this case, \eqref{eq:fibre-Pd},
\eqref{eq:Prym=ker-ramified} and \eqref{eq:det-norm} show that the
fibre of the Hitchin map $h:\mathcal{M}\to\mathcal A$ over a generic point $s\in\mathcal A'$ is naturally
identified with $\Prym_\pi(X_s)$ and not just a torsor over it, and
analogously for the fibre of $h'$.  Hence, in view of the observations
preceding the statement of the theorem, we have the desired conclusion
if we show that the base points of these Pryms form a Lagrangian
section $\sigma$ of $h$.

This is obtained, similarly to Hitchin's construction
\cite{hitchin:1992} in the non-parabolic case, as follows\footnote{The
  authors thank Laura Fredrickson for pointing out a mistake at this
  point of an earlier version of the paper.}. Given
$(s_2,\ldots,s_n)\in\mathcal A$, consider the parabolic Higgs bundle
$(V,\vp)$ where
\begin{equation}\label{eq:Hitsect-bundle}
V=K(D)^{1-n}\oplus\cdots\oplus K(D)^{-1}\oplus\cO_X,
\end{equation}
with the Higgs field defined, in this decomposition, by
\begin{equation}\label{eq:Hitsect-field}
\varphi=\left(\begin{array}{ccccc}0 & 1 & \cdots & 0 & 0 \\\vdots & \ddots & \ddots & \vdots & \vdots \\0 & 0 & \ddots & 1 & 0 \\0 & 0 & \cdots & 0 & 1 \\s_n & s_{n-1} & \cdots & s_2 & 0\end{array}\right)
\end{equation} 
and such that the full flag of $V_p$ over each $p\in D$ is given by 
\begin{equation}\label{eq:Hitsect-flag}
V_p \supsetneq K(D)_p^{2-n}\oplus\cdots\oplus \cO_{X,p} \supsetneq \cdots \supsetneq
 \cO_{X,p}\supsetneq \{0\},
\end{equation} 
with the weights determined by $\cM$. It is clear that $\varphi$ is
strongly parabolic. Moreover, for any choice of weights $\bm\alpha$,
the stability of $(V,\vp)$ follows by a straightforward adaptation of
the corresponding argument in \cite{hitchin:1992}.  Hence the map
$\sigma:\mathcal A\to\mathcal M,\;(s_2,\ldots,s_n)\mapsto(V,\vp)$, where
the parabolic Higgs bundle $(V,\vp)$ is defined by
\eqref{eq:Hitsect-bundle}--\eqref{eq:Hitsect-flag}, provides a section
of $h$. This section is Lagrangian for the same reason as in the non-parabolic
case, namely that the underlying bundle is fixed as $(s_2,\ldots,s_n)$
moves in the Hitchin base. Finally, for each $s\in\mathcal A'$, $V$ in
\eqref{eq:Hitsect-bundle} is such that $V\cong\pi_*\mathcal O_{X_s}$,
so $\sigma$ is indeed a Lagrangian section through the base points
$\mathcal O_{X_s}\in\Prym_\pi(X_s)$.

Now, for any other line bundle $\Lambda$ of any degree, use
Corollary \ref{cor:isomdifferentdegreesandweights-fullflag} to get an
isomorphism
$\cM_\Lambda^{\bm\a'}(\SL(n,\C))\cong
\cM_{K(D)^{-n(n-1)/2}}^{\bm\a}(\SL(n,\C))$ (since we are assuming full
flags, the hypotheses of Corollary
\ref{cor:isomdifferentdegreesandweights-fullflag} are satisfied). Clearly this map
is in fact an isomorphism of the corresponding Hitchin systems, and
descends to the $\PGL(n,\C)$-Hitchin systems, giving
us in particular the desired identifications of the Hitchin fibers as
dual abelian varieties.
\endproof

\begin{remark}
  This fits with a general
  phenomenon in SYZ-mirror symmetry, where if torus fibrations admit a
  Lagrangian section, then the $B$-field is unnecessary for the
  symmetry to work out; see for example Hitchin~\cite{hitchin:2001} or
  Polishchuk~\cite{polishchuk:2003}.
  In the more general version involving a $B$-field, the
  identification of Hitchin fibers as dual abelian varieties comes
  about through a choice of trivialization of the restriction of the
  gerbe. Moreover, if there is a canonical coherent choice of
  trivialization of the gerbe in all fibers, the ``gerby'' duality
  (see Hausel--Thaddeus \cite[p.~202]{hausel-thaddeus:2003} for its
  definition) follows from the usual one described here. This would be
  the case if, for example, the gerbe were known to be trivial.
\end{remark}

\subsection{The (stringy) $E$-polynomial}
\label{sec:stringy-E-polynomial}

Let $M$ be a \emph{semiprojective} variety (see
\cite{hausel-rodriguez-villegas:2015}). This means that $M$ is
quasi-projective and that (i) it carries an algebraic $\C^*$-action
such that for any point $p$ in $M$, the limit of the $\C^*$-orbit
$(t\cdot p)_{t\in\C^*}$ when $t$ goes to $0$ exists in $M$, and (ii)
the subvarieties of $M$ of fixed points under $\C^*$ are compact. Then
by \cite[Corollary 1.3.2]{hausel-rodriguez-villegas:2015} if $M$ is
smooth, the (compactly supported) cohomology of $M$ is pure. Hence its
\emph{$E$-polynomial} is given by
\[
E(M)=\sum_{p,q=0}^{\dim(M)}(-1)^{p+q}h_c^{p,q}(M)u^pv^q,
\]
where $h_c^{p,q}(M)=\dim H^{p,q}_c(M,\C)$. Also, if $M$ has an action
of a group $\Gamma$, we let $E(M)^\Gamma$ denote the
\emph{$\Gamma$-invariant $E$-polynomial}, i.e.,
\begin{equation}\label{eq:Epolynomial-mixed-Gamma-inv}
  E(M)^{\Gamma_n} = \sum_{p,q=0}^{\dim(M)}(-1)^{p+q}\dim H_c^{p,q}(M)^\Gamma u^pv^q,
\end{equation}
where $H_c^{p,q}(M)^\Gamma \subset H_c^{p,q}(M)$ is the
$\Gamma$-invariant subspace.

The main motivational example for the definition of semiprojective
varieties comes precisely from the moduli spaces of Higgs
bundles. Indeed, they carry an algebraic $\C^*$-action, also in the
strongly parabolic case, defined by
\begin{equation}\label{eq:C*-action}
t\cdot(V,\vp)=(V,t\vp),\ t\in\C^*.
\end{equation}
The following proposition describes two well-known fundamental
properties of this action, which in particular show that $\cM$ is
semiprojective.

\begin{proposition}\label{prop:C*-action-prop} The $\C^*$-action on
  $\cM$ verifies the following properties.
\begin{enumerate}
\item For any point $(V,\vp)\in\cM$, the limit $\lim_{t\to 0}(V,t\vp)$
  exists in $\cM$ and is a fixed point.
\item The subvarieties of $\cM$ of fixed points are compact. 
\end{enumerate}
Therefore, $\cM$ is semiprojective.
\end{proposition}
\proof This is well known. The first item follows from the properness
of the Hitchin map \eqref{eq:Hitchinmap}, just as in \cite[Corollary
9.20]{simpson:1994b}. Regarding the second item, the $\C^*$-fixed
points are precisely the critical points of the real function
$f:\cM\to\R_{\geq 0}$ given by the $L^2$-norm of the Higgs field (see
Proposition 3.3 of \cite{garcia-prada-gothen-munoz:2007}). By Theorem
4.1 (a) of \cite{boden-yokogawa:1996}, $f$ is proper, hence (2) also
follows.
\endproof

The moduli space $\cM/\Gn$ has orbifold singularities and, following
Hausel-Thaddeus \cite{hausel-thaddeus:2003} (in turn based on Batyrev--Dais
\cite{batyrev-dais:1996} and Batyrev \cite{batyrev:1998}), we consider
the 
\emph{stringy $E$-polynomial} of $\cM/\Gn$, defined as 
\begin{equation}\label{eq:stringyEpolynomial}
E_{\mathrm{st}}(\cM/\Gn)=\sum_{\gamma\in\Gn}E(\cM^\gamma)^{\Gn}(uv)^{F(\gamma)},
\end{equation}
where the various objects on the right hand side will be defined in
the following. We note in passing that this stringy $E$-polynomial
encodes the Chen--Ruan cohomology of $\cM/\Gn$ as an orbifold
\cite{chen-ruan:2004}.

The subspace $\cM^\gamma\subset\cM$ is the locus pointwise fixed by
$\gamma$. Since it is abelian, $\Gn$ preserves $\cM^\gamma$ and then
$E(\cM^\gamma)^{\Gn}$ is defined as in
\eqref{eq:Epolynomial-mixed-Gamma-inv}. 
The \emph{fermionic shift} $F(\gamma)$ is defined as follows: given
$p\in\cM^\gamma$, the element $\gamma$ acts on the tangent space
$T_p\cM$ with eigenvalues
$\lambda_1,\ldots,\lambda_{\dim(\cM)}$. Since $\Gn$ is finite, these
are roots of the unity, hence we can write $\lambda_j=e^{2\pi i w_j}$,
with $0\leq w_j<1$ and $w_j\in\mathbb{Q}$. The {fermionic shift} is
the number
\begin{equation}\label{eq:fermionicshift}
F(p,\gamma)=\sum_{j=1}^{\dim(\cM)}w_j.
\end{equation}
Clearly it is constant along the connected component of $\cM^\gamma$
containing $p$.  In general the fermionic shift is just a rational
number but, in our case, we can be much more precise. Recall that
$\gamma$ acts by
$\gamma\cdot (V,\vp)=(V\otimes L_\gamma,\vp\otimes\Id_{L_\gamma})$,
where $L_\gamma$ is the corresponding $n$-torsion line bundle. Recall
that $\cN$ denotes the moduli space of parabolic vector bundles with
fixed determinant $\Lambda$ and the same parabolic structure as the
one considered in $\cM$. Then $\gamma$ acts by diffeomorphisms on
$\cN$, hence acts by symplectomorphisms on the cotangent bundle
$T^*\cN$, which is an open and dense subspace of $\cM$, so $\gamma$
acts by symplectomorphisms on $\cM$. It follows that for each
eigenvalue $\lambda_j$, $\lambda_j^{-1}$ is also an eigenvalue. Since
$\gamma$ acts trivially on the subspace
$T_p\cM^\gamma\subseteq T_p\cM$, we conclude that
\begin{equation}\label{eq:fermionic=rkN/2}
F(p,\gamma)
=\sum_{l=1}^{\dim(N_p\cM^\gamma)/2}(w_l+1-w_l)=\frac{\dim(N_p\cM^\gamma)}{2}.
\end{equation}
where $N_p\cM^\gamma\subseteq T_p\cM$ denotes the normal bundle to
$\cM^\gamma$ at $p$.  We have already observed that, in general,
$F(p,\gamma)$ only depends on the connected component of $\cM^\gamma$
containing $p$. We shall in fact see in Section \ref{sec:stringy
  E-pol} that $\cM^\gamma$ is non-connected, but we shall also
conclude directly (see \eqref{eq:comput of fermionic} below) that the
value of $F(p,\gamma)$ is really independent of $p$, thus independent
of the component where it lies. That is the reason why we just wrote
$F(\gamma)$ in the definition \eqref{eq:stringyEpolynomial} of the
stringy $E$-polynomial.

\begin{remark}\label{rmk:F(e)=0}
  Note that $F(e)=0$ where $e$ is the trivial element of $\Gn$. Thus
  the stringy $E$-polynomial of a smooth variety coincides with the
  usual one. In particular we have that
  $E(\mathcal{M}) = E_{\mathrm{st}}(\mathcal{M})$.
\end{remark}

We conclude this section by pointing out that the $E$-polynomials are
independent of the weights $\bm\alpha$, as long as these are
generic. This will be useful in our later calculations (specifically,
in the proof of Proposition~\ref{prop:E-pol-variant-part}) since it
allows us to make simplifying assumptions on the weights. Everything
follows from the work of Thaddeus \cite{thaddeus:2002}, who studied
how the moduli space of parabolic Higgs bundles changes under wall
crossing of the parabolic weights. It is immediate from his
description that the Hodge numbers of the moduli space are unchanged
under wall crossing. We shall need a $\Gn$-equivariant version of this
result. This also follows from Thaddeus' description, which we now
briefly recall.

Let $\bm{\alpha}$ belong to only one wall in the space of parabolic
weights and consider weights $\bm\alpha^-$ and $\bm\alpha^+$ in the two adjacent
chambers. For brevity write $\mathcal{M}^\pm$ for either of the moduli
spaces $\mathcal{M}^{\bm\alpha^\pm}_{\Lambda}(\SL(n,\C))$ and
$\mathcal{M}^{\bm\alpha^\pm}_{d}(\GL(n,\C))$ (everything in this
section applies to both of these).
There are \emph{flip loci}
$\mathcal{S}^\pm\subset \mathcal{M}^{\pm}$ which correspond
to those parabolic Higgs bundles which are $\bm\alpha_{\pm}$-stable and
$\bm\alpha_{\mp}$-unstable. In the following, write $\mathbf{V}=(V,\vp)$
for a parabolic Higgs bundle. Points of $\mathcal{S}^-$ correspond to
parabolic Higgs bundles $\mathbf{V}$ which are non-split extensions
\begin{equation}
  \label{eq:S-minus}
  0 \to \mathbf{V}^+ \to \mathbf{V} \to \mathbf{V}^- \to 0,
\end{equation}
of parabolic Higgs bundles, where $\mathbf{V}^\pm$ are
stable with respect to the parabolic weights induced by
$\bm\alpha^\pm$. There is an analogous description of $\mathcal{S}^+$.
Thus there is a natural identification
\[
  g\colon \mathcal{M}^-\setminus \mathcal{S}^- \xra{\cong}
  \mathcal{M}^+\setminus \mathcal{S}^+.
\]
Denote by
$\pi^\pm\colon \tilde{\mathcal{M}}^{\pm} \to
{\mathcal{M}}^\pm$ the blow-ups of $\mathcal{M}^\pm$ along
$\mathcal{S}^\pm$ and by
$\mathcal{E}^\pm\subset \tilde{\mathcal{M}}^\pm$ the exceptional
divisors. Thaddeus \cite[(6.2)]{thaddeus:2002} proves that there is an
isomorphism
\begin{equation}
  \label{eq:M-tilde-iso}
  \tilde{g}\colon \tilde{\mathcal{M}}^- \xra{\cong}
    \tilde{\mathcal{M}}^+
\end{equation}
which restricts to an isomorphism
$\mathcal{E}^-\xra{\cong}\mathcal{E}^+$ of the exceptional divisors
and coincides with $g$ on their complement. It is a standard fact
about blow-ups that the cohomology groups of $\mathcal{M}^\pm$ inject
into the cohomology groups of $\tilde{\mathcal{M}}^\pm$ and from
\eqref{eq:M-tilde-iso} it follows that $\tilde{g}$ induces isomorphisms
\begin{equation}
  \label{eq:pm-pq-iso}
  H^{p,q}_c(\mathcal{M}^-) \cong  H^{p,q}_c(\mathcal{M}^+),
\end{equation}
considering these cohomology groups as subspaces of
$H^{p,q}_c(\tilde{\mathcal{M}}^\pm)$. Thus, for generic $\bm\alpha$, the
$E$-polynomials of $\mathcal{M}_d^{\bm\alpha}(\GL(n,\C))$ and
$\mathcal{M}_\Lambda^{\bm\alpha}(\SL(n,\C))$ are independent of
$\bm\a$.

In view of what we have said so far, it is now easy to prove the
following.

\begin{proposition}
  Let $\Gamma_n=\Jac_{[n]}(X)$ act on the moduli space of parabolic
  Higgs bundles by the action defined in \eqref{eq:Gamma_n action}.
  The isomorphism $\tilde{g}$ of \eqref{eq:M-tilde-iso} is
  equivariant with respect to this action.  Consequently, the
  isomorphism \eqref{eq:pm-pq-iso} is also $\Gn$-equivariant.
\end{proposition}

\begin{proof}
  The basic observation is that the action of $\Gamma_n$ preserves
  $\mathcal{S}^\pm\subset \mathcal{M}^\pm$; this follows from the
  description of $\mathcal{S}^-$ (and the analogous description of
  $\mathcal{S}^+$) as corresponding to extensions of the form
  \eqref{eq:S-minus}. Hence the $\Gamma_n$-actions lift to the
  blow-ups $\tilde{\mathcal{M}}^\pm$ (as follows from the universal
  property of the blow-up, Hartshorne
  \cite[Cor.~II.7.15]{hartshorne:1977}). Moreover, the
  restriction of $\tilde{g}$ to the open dense subset
  $\tilde{\mathcal{M}}^-\setminus\mathcal{E}^-\subset
  \tilde{\mathcal{M}}^-$ is just $g$, which is $\Gamma_n$-equivariant
  by our initial basic observation. It follows that $\tilde{g}$
  is $\Gamma_n$-equivariant as claimed.
\end{proof}

\begin{corollary}
  \label{cor:alpha-independent-E-poly} 
  Assume that $\bm\alpha$ is generic and let $\mathcal{M}^{\bm\alpha}$
  denote either $\mathcal{M}_d^{\bm\alpha}(\GL(n,\C))$ or
  $\mathcal{M}_\Lambda^{\bm\alpha}(\SL(n,\C))$. Then the compactly
  supported Dolbeault cohomology of $\mathcal{M}^{\bm\alpha}$ is
  independent of $\bm\alpha$ as a $\Gamma_n$-module. Thus the
  $E$-polynomial $E(\mathcal{M}^{\bm\alpha})$ and the
  $\Gamma_n$-invariant $E$-polynomial
  $E(\mathcal{M}^{\bm\alpha})^{\Gamma_n}$ are both independent of
  $\bm\alpha$. \qed
\end{corollary}

\begin{remark}
  \label{rem:alpha-independent-PGLn}
  We shall see that the \emph{stringy} $E$-polynomial of the moduli
  space of parabolic $\PGL(n,\C)$-Higgs bundles is also independent of
  $\bm\alpha$. Indeed, it will follow from the description given in
  Theorem~\ref{thm:fixedpointlocus1} below that for any
  $e\neq\gamma\in\Gamma_n$, the parabolic Higgs bundles in the fixed
  locus
  $(\mathcal{M}^{\bm\alpha})^\gamma\subset\mathcal{M}^{\bm\alpha}$ are
  $\bm\alpha$-semistable for any value of $\bm\alpha$. In other words
  $(\mathcal{M}^{\bm\alpha})^\gamma$ does not intersect the flip locus
  $\mathcal{S}^\pm$ and thus $\tilde{g}$ from (\ref{eq:M-tilde-iso})
  restricts to a $\Gamma_n$-equivariant isomorphism. Thus all the terms
  in the definition \eqref{eq:stringyEpolynomial} of the stringy
  $E$-polynomial are independent of $\bm\alpha$.
\end{remark}

\subsection{Topological mirror symmetry and the main result}
\label{sec:topological-MS}

The topological mirror symmetry conjecture of Hausel--Thaddeus says
that the stringy $E$-polynomials of the mirror partners $\mathcal{M}$
and $\mathcal{M}/\Gamma_n$ should agree. Since the SYZ mirror symmetry
statement is really about the de Rham moduli spaces, rather than the
Dolbeault moduli spaces, so is the topological
mirror symmetry conjecture (see Remark~\ref{rem:SYZ-de-Rham-Dolbeault}). On the other hand, it is the rich
algebraic geometry of the Higgs bundle moduli spaces and, in
particular, the fact that it carries a $\C^*$-action which allows Hausel
and Thaddeus \cite{hausel-thaddeus:2003} to prove the equality of the
$E$-polynomials in the non-parabolic case. This suffices because they also prove that the de Rham and
Dolbeault moduli spaces have the same $E$-polynomials. The proof of
this latter result uses
that the Dolbeault and de Rham moduli spaces live in a family, the
\emph{Hodge moduli space}, which parametrizes so-called
$\lambda$-connections.  This moduli space fibers over the affine line
$\C$ with fibers away from zero all isomorphic to the de Rham moduli
space and degenerating to the Dolbeault moduli space over
zero. Parabolic $\lambda$-connections and the corresponding moduli
spaces were constructed and studied by Alfaya--Gómez
\cite{alfaya-gomez:2016}, and their results provide the necessary input
for applying the arguments of Hausel--Thaddeus
\cite[Sec.~6]{hausel-thaddeus:2003} (cf.\ Hausel--Rodriguez-Villegas
\cite[Cor.~1.3.3]{hausel-rodriguez-villegas:2015}) directly in the parabolic
situation. Thus the parabolic de Rham moduli spaces have the same
$E$-polynomials as the moduli spaces of parabolic Higgs bundles, and
we can exclusively work with the latter for the remainder of the paper.

The topological mirror symmetry conjecture can now be stated in
terms of the Higgs bundle moduli spaces as follows.
\begin{conjecture}[Hausel--Thaddeus \cite{hausel-thaddeus:2001,hausel-thaddeus:2003}]
\label{conj:conjecture1}
For any rank $n$, any line bundle $\Lambda$ and any system
of generic weights $\bm\a$, the equality of $E$-polynomials
\begin{equation}
  \label{eq:equalityE-pol-deg0-Dol}
  E(\cM^{\bm\alpha}_\Lambda)=E_{\mathrm{st}}(\cM^{\bm\alpha}_\Lambda/\Gn)
\end{equation}
holds.
\end{conjecture}


Our main result states that this is true for $n=2,3$.

\begin{theorem}\label{thm:main1}
If $n=2,3$, then Conjecture \ref{conj:conjecture1} holds.
\end{theorem}
\proof 
We follow the strategy of \cite{hausel-thaddeus:2003} which we now explain.
From the definition of the stringy $E$-polynomial \eqref{eq:stringyEpolynomial} of $\cM/\Gn$ and from Remark \ref{rmk:F(e)=0}, we have that
\[
E_{\mathrm{st}}(\cM/\Gn)=E(\cM)^{\Gn}+\sum_{\gamma\neq e}E(\cM^\gamma)^{\Gn}(uv)^{F(\gamma)}.
\] 
On the other hand, let $E(\cM)^\var$ denote the \emph{variant} part of
$E(\cM)$ in Hausel and Thaddeus' terminology. It is defined
analogously to $E(\cM)$ but the coefficients are given by subtracting
the dimensions of the $\Gn$-invariant subspaces, i.e.,
\[
E(\cM)=E(\cM)^{\Gn}+E(\cM)^\var.
\]
Hence \eqref{eq:equalityE-pol-deg0-Dol} is equivalent to
\begin{equation}\label{eq:conjecture2}
E(\cM)^\var=\sum_{\gamma\neq e}E(\cM^\gamma)^{\Gn}(uv)^{F(\gamma)}.
\end{equation}
Now, Theorems \ref{thm:varpartn23} and \ref{thm:main2} below imply that \eqref{eq:conjecture2} holds for any $n=2,3$, proving Theorem \ref{thm:main1}.
\endproof


Thus the following two theorems complete the proof of Theorem \ref{thm:main1}. Here $\cF_{(1,1,\ldots,1)}$ denotes the subspace of $\cM$ consisting of subvarieties of fixed points of the $\C^*$-action \eqref{eq:C*-action} of type $(1,1,\ldots,1)$, to be properly defined in the following section (see in particular \eqref{eq:M=affinebundlefixed}), and $E(\cF_{(1,1,\ldots,1)})^\var$ is the variant part of the corresponding $E$-polynomial.

\begin{theorem}\label{thm:varpartn23}
Let $n=2,3$. For any system of generic weights $\bm\a$, and any line bundle
$\Lambda$, we have $E(\cM)^\var=(uv)^{\dim(\cM)/2}E(\cF_{(1,1,\ldots,1)})^\var$.
\end{theorem}
\proof
For $n=2$, this follows from comparing Holla \cite[Theorem 5.23]{holla:2000} and Nitsure \cite[Proposition 3.11]{nitsure:1986} and using the argument of Atiyah--Bott \cite[Prop.~9.7]{atiyah-bott:1982}; see also \cite[Remark 3.11]{nitsure:1986} and \cite[Remark 10.1]{garcia-prada-gothen-munoz:2007}. For $n=3$, it follows from \cite[Theorem 12.22]{garcia-prada-gothen-munoz:2007}; see also Remarks 12.17 and 12.19 of loc. cit..
\endproof

We shall say that a result \emph{holds for small weights} if there is
an $\epsilon>0$ such that the result holds for any system of weights
$\bm\a$ with $\alpha_i(p)<\epsilon$ for all $i$ and $p$. 

\begin{theorem}\label{thm:main2}
For any $n$ prime, any system of small generic weights $\bm\a$, and any line bundle
$\Lambda$, we have
\begin{equation}\label{eq:polynom0}
(uv)^{\dim(\cM)/2}E(\cF_{(1,1,\ldots,1)})^\var=\sum_{\gamma\neq e}E(\cM^\gamma)^{\Gn}(uv)^{F(\gamma)}
\end{equation}
and both sides are equal to
\begin{equation}\label{eq:polynom}
\frac{n^{2g}-1}{n}(n!)^{|D|}(uv)^{(n^2-1)(g-1)+|D|n(n-1)/2}((1-u)(1-v))^{(n-1)(g-1)}.
\end{equation}
Moreover, if $n=2,3$, the result holds for any system of generic weights.
\end{theorem}

\begin{remark}
When $n=2$, the polynomial \eqref{eq:polynom} is equivalent to the one which
appears in \cite{hausel-thaddeus:2001}, the difference in sign
being due to different conventions.
\end{remark}

\begin{remark}
  We shall conclude directly in Section \ref{sec:stringy E-pol} that
  the right-hand side of \eqref{eq:polynom0} is independent of the
  generic weights (see also Remark \ref{rem:alpha-independent-PGLn}),
  so that the assumption on small weights is only needed to compute
  $E(\cF_{(1,1,\ldots,1)})^\var$. Specifically, it is  used in the proof
  of Proposition \ref{prop:E-pol-variant-part} below. Now, it follows
  from Corollary \ref{cor:alpha-independent-E-poly} that the
  polynomial $E(\cM)^\var$ is independent of the generic weights, so
  if Theorem \ref{thm:varpartn23} were known to be true for any prime
  $n$, then we could remove this small weights assumption from Theorem
  \ref{thm:main2}. The only obstacle for having a proof of Theorem
  \ref{thm:main1} for any $n$ prime is thus the fact that Theorem
  \ref{thm:varpartn23} is not known to hold for such $n$.
\end{remark}

The remaining part of the paper will be dedicated to the proof of
Theorem \ref{thm:main2}, which follows from Proposition
\ref{prop:E-pol-variant-part}, Corollary
\ref{cor:E-pol-variant-part-n23} and Proposition
\ref{prop:stringyE-poly} below.  Again we follow the arguments of
\cite{hausel-thaddeus:2003}. We shall prove that both
$(uv)^{\dim(\cM)/2}E(\cF_{(1,1,\ldots,1)})^\var$ and
$\sum_{\gamma\neq e}E(\cM^\gamma)^{\Gn}(uv)^{F(\gamma)}$ are equal to
the given polynomial (in the former case for small weights if
$n>3$). The proofs of these equalities are completely independent of
each other. The case of
$(uv)^{\dim(\cM)/2}E(\cF_{(1,1,\ldots,1)})^\var$ will be treated in
Section \ref{sec:(11...1)}, while the case of
$\sum_{\gamma\neq e}E(\cM^\gamma)^{\Gn}(uv)^{F(\gamma)}$ is going to
be dealt with in Section \ref{sec:stringy E-pol}. Section
\ref{sec:unramified-Norm,Prym,actions} is an independent section,
containing some results on Prym varieties of unramified covers, which
are needed in Section \ref{sec:stringy E-pol}.


\section{The polynomial $(uv)^{\dim(\cM)/2}E(\cF_{(1,1,\ldots,1)})^\var$}\label{sec:(11...1)}


The $\C^*$-action \eqref{eq:C*-action} on the moduli space $\cM$ is a
fundamental tool on the study of its geometry and topology. In
particular the cohomology of $\cM$ is completely determined by the
cohomology of the subvarieties of fixed points, hence so is the
$E$-polynomial of $\cM$. Here we aim to compute the $E$-polynomial of a certain subspace of the fixed point loci of the $\C^*$-action,
relevant for Theorem \ref{thm:main2}. In the next subsections, we describe these fixed point locus.

\subsection{The fixed points of the $\C^*$-action}

Here we shall consider the fixed point subvarieties of the
$\C^*$-action \eqref{eq:C*-action}. From Proposition
\ref{prop:C*-action-prop} we know that these are compact, but now we
need a more explicit description of the fixed points. This is provided
by the following result due to Simpson (see \cite[Theorem
8]{simpson:1990}).

\begin{proposition}\label{prop:fixedpoints-Simpson}
A stable parabolic $\SL(n,\C)$-Higgs bundle $(V,\vp)\in\cM$ is a fixed point under $\C^*$ if and only if either
\begin{enumerate}
\item $\vp\equiv 0$, or
\item $V$ admits a decomposition $V\cong \bigoplus_{j=1}^lV_j$ such that the following hold:
\begin{itemize}
\item the subbundles $V_j$ are parabolic and the decomposition $V\cong \bigoplus_{j=1}^lV_j$ is compatible with the parabolic structure, i.e., at every point $p\in D$, every subspace $V_{p,i}$ is a direct sum of fibers at $p$ of certain subbundles $V_j$. 
\item the Higgs field splits as $\vp=\sum_{j=1}^l\vp_j$, with $\vp_j:V_j\to V_{j+1}\otimes K(D)$ non-zero for all $j=1,\ldots,l-1$, and $\vp_l\equiv 0$. 
\end{itemize}
\end{enumerate}
\end{proposition}

A parabolic Higgs bundle of the kind described in the preceding
proposition is called a \emph{Hodge bundle}. Note that we can
include the ones of the form $(V,0)$ in point (2) by taking $l=1$,
however it will be convenient for us to distinguish the two kinds of
fixed points notationally.
 
\begin{definition}
 A fixed point with non-vanishing Higgs field is said to be of \emph{type $(n_1,n_2,\ldots,n_l)$}, with $\sum n_j=n$, if $\rk(V_j)=n_j$, for all $j$. Denote by $\cF_{(n_1,n_2,\ldots,n_l)}$ the union of the subvarieties of $\cM$ of all fixed points of type $(n_1,n_2,\ldots,n_l)$.
\end{definition}

As is well known, it follows from Bialynicki-Birula stratification
associated to the $\C^*$-action that the cohomology of $\cM$ is
determined by the cohomology of all fixed point subvarieties of the
$\C^*$-action. Indeed, the $\C^*$-flows gives rise to Zariski
locally trivial affine bundles, with fibre $\C^{\dim(\cM)/2}$, over
the disjoint union of all $\cF_{(n_1,n_2,\ldots,n_l)}$ together with
$\cN$. This follows by Proposition \ref{prop:C*-action-prop} (1), and
the projection of these affine bundles is just taking the limit of the
flow when $t$ goes to $0$.  Since the $E$-polynomial is additive with
respect to disjoint unions and multiplicative with respect to locally
trivial fibrations in the Zariski topology, we consequently have that
\begin{equation}\label{eq:M=affinebundlefixed}
E(\cM)=(uv)^{\dim(\cM)/2}\left(E(\cN)+\sum_{(n_1,n_2,\ldots,n_l)}E(\cF_{(n_1,n_2,\ldots,n_l)})\right).
\end{equation}
All $\cF_{(n_1,n_2,\ldots,n_l)}$ and $\cN$ are smooth and projective
so we can consider their usual $E$-polynomials.

According to Theorem \ref{thm:main2}, the relevant subvarieties to be considered are the ones corresponding to type $(1,1,\ldots,1)$, that is
$\cF_{(1,1,\ldots,1)}$.

\subsection{The subvarieties $\cF_{(1,1,\ldots,1)}$}

Let $n$ be a prime number.  Our next task is to obtain a geometric
description of the subspace $\cF_{(1,1,\ldots,1)}$. If $(V,\vp)$
represents a fixed point of the $\C^*$-action of type $(1,1,\ldots,1)$
then
\begin{equation}\label{eq:11...1}
V=\bigoplus_{j=1}^n L_j\qquad\text{ and }\qquad \vp=\sum_{j=1}^{n-1}\vp_j,\quad \vp_j:L_j\to L_{j+1}\otimes K(D),\ \vp_n\equiv 0.
\end{equation}
Since in $\cM$ we always have fixed determinant $\Lambda$, then 
\begin{equation}\label{eq:MHS-fixed det}
\prod L_j\cong\Lambda.
\end{equation} The subspace $\cF_{(1,1,\ldots,1)}$ is decomposed into connected components which can be labeled by the topological data coming from decomposition \eqref{eq:11...1}, namely the degrees of the bundles $L_j$ and the way the weights are distributed among them at each point of $D$. Actually, instead of using the degrees of the bundles $L_j$, we shall opt for a slight variation of this.

Over each $p\in D$, we have the corresponding parabolic structure
\begin{equation}\label{eq:filtration-11...1}
V_p=V_{p,1} \supsetneq V_{p,2} \supsetneq \cdots \supsetneq V_{p,n}\supsetneq \{0\},
\qquad 0\leq \a_1(p)< \cdots <\a_n(p) < 1.
\end{equation}
By Proposition \ref{prop:fixedpoints-Simpson}, each $L_j$ is a parabolic subbundle of $V$ and the decomposition \eqref{eq:11...1} is compatible with the parabolic structure \eqref{eq:filtration-11...1}. The filtration of the fibre $L_{j,p}$ of $L_j$ at $p$ is of course trivial 
\begin{equation}\label{eq:trivfiltL}
L_{j,p}\supsetneq\{0\},
\end{equation}
and the corresponding weight $\beta_j(p)$ assigned to $L_{j,p}$ is $\beta_j(p)=\a_i(p)$ where $i$ is such that $L_{j,p}\subseteq  V_{p,i}$ but $L_{j,p}\not\subseteq  V_{p,i+1}$; this is precisely the condition coming from \eqref{eq:weight-subbundle}. Since there are $n$ line subbundles and the filtration \eqref{eq:filtration-11...1} has length $n$, we see that \eqref{eq:filtration-11...1} is determined by a distribution of the weights at $p$ among the fibers  of the line subbundles $L_j$ at $p$. Precisely, $V_{p,n}=L_{j,p}$  where $j$ is such that $\beta_j(p)=\a_n(p)$ and, for $i<n$, $V_{p,i}=V_{p,i+1}\oplus L_{j',p}$ with $j'$ such that $\beta_{j'}(p)=\a_i(p)$. Such distribution of the $n$ weights at $p$ is provided by a permutation of the set $\{1,\ldots,n\}$, so by an element $\varpi_n(p)$ of the symmetric group $S_n$. Write such permutation by a word 
\[\varpi_n(p)=a_1(p) a_2(p)\dots a_n(p)\in S_n\] with $a_j(p)\in\{1,\ldots,n\}$, where this means that we assign the weight $\a_{a_j(p)}(p)$ to the fibre $L_{j,p}$.
The conclusion is that the parabolic structure on $V=\bigoplus_{j=1}^n L_j$ is determined by an element 
\begin{equation}\label{eq:permutations}
\varpi_n=(\varpi_n(p_1),\ldots,\varpi_n(p_{|D|}))\in S_n^{|D|}.
\end{equation}

Now we have to see how the Higgs field comes into play. It is given by \eqref{eq:11...1}, so $\vp_j\in H^0(X,\SPH(L_j,L_{j+1})\otimes K(D))$ for every $j$.
The residue of $\vp$ at $p\in D$ is given, according to the decomposition  \eqref{eq:11...1} of $V$, by 
\[\vp_p=\left(\begin{array}{ccccc}0 & 0 & \dots & 0 & 0 \\ \vp_{1,p} & 0 & \dots & 0 & 0 \\ \vdots & \vdots & \ddots & \vdots & \vdots \\0 & 0 & \dots & \vp_{n-1,p} & 0\end{array}\right).\] 
Suppose $\varpi_n(p)=a_1(p) a_2(p)\dots a_n(p)$. Since $\vp$ is strongly parabolic, it follows from \eqref{eq:trivfiltL} that if $a_j(p)>a_{j+1}(p)$ then $\vp_{j,p}=0$. Thus
\[
a_j(p)>a_{j+1}(p)\Longrightarrow  \vp_j\in H^0(X,\Hom(L_j,L_{j+1})\otimes K(D-p)).
\]
For each $j=1,\ldots,n-1$, define the subdivisor of $D$
\[S_j(\varpi_n)=\{p\in D\st a_j(p)>a_{j+1}(p)\}\subseteq D,\]
so that
\[
\vp_j\in H^0(X,\Hom(L_j,L_{j+1})\otimes K(D-S_j(\varpi_n))).
\]

Let 
\[M_j=L_j^{-1}L_{j+1}K(D-S_j(\varpi_n))\] and write
\begin{equation}\label{eq:degMj}
m_j=\deg(M_j)=-d_j+d_{j+1}+2g-2+|D|-s_j(\varpi_n)\geq 0.
\end{equation} 
with $d_j=\deg(L_j)$ and $s_j(\varpi_n)=|S_j(\varpi_n)|$, the cardinal of $S_j(\varpi_n)$.
By \eqref{eq:MHS-fixed det}, 
\begin{equation}\label{eq:prodMj}
\prod_{j=1}^{n-1}M_j^j\cong L_n^n\Lambda^{-1}K^{\frac{n(n-1)}{2}}\left(\frac{n(n-1)}{2}D-\sum_{j=1}^{n-1}jS_j(\varpi_n)\right)
\end{equation} 
and this implies
\begin{equation}\label{eq:constrain on degMj}
\begin{cases}
\displaystyle d+\sum_{j=1}^{n-1}j(m_j+s_j(\varpi_n))\equiv 0\hspace{-.25cm}\pmod n,&\text{ if }\quad n\geq 3\\
d+m_1+s_1(\varpi_2)-|D|\equiv 0\hspace{-.25cm} \pmod 2,&\text{ if }\quad n=2.
\end{cases}
\end{equation}
Clearly the collection $(m_j)_j$ determines the collection $(d_j)_j$ and vice-versa through \eqref{eq:degMj} and \eqref{eq:prodMj}.

The proper $\vp$-invariant subbundles of $V$ are the ones of the form $V_l=\bigoplus_{j=l}^nL_j$, for $2\leq l\leq n$.
The stability condition $\pmu(V_l)<\pmu(V)$ (cf. Definition \ref{def:stabcond}) for the subbundle $V_l$ reads as 
\begin{equation}\label{eq:stab cond-VHS}
\begin{split}(n-l+1)\sum_{j=1}^{l-1}jm_j+(l-1)\sum_{j=l}^{n-1}(n-j)m_j&< \sum_{p\in D}\left(\sum_{i=1}^n(n-l+1)\alpha_i(p)-n\sum_{j=l}^n\alpha_{a_j(p)}(p)\right)+\\
&+(g-1+|D|/2)n(n-l+1)(l-1)-\\
&-(n-l+1)\sum_{j=1}^{l-1}js_j(\varpi_n)-(l-1)\sum_{j=l}^{n-1}(n-j)s_j(\varpi_n).
\end{split}
\end{equation}

Given $\varpi_n$ as in \eqref{eq:permutations} and $m_1,\ldots,m_{n-1}$ non-negative integers such that \eqref{eq:constrain on degMj} and \eqref{eq:stab cond-VHS} hold, denote by $\cF_{(1,1,\ldots,1)}(\varpi_n,m_1,\ldots,m_{n-1})$ be the subspace of $\cF_{(1,1,\ldots,1)}$ determined by the given numerical/topological data.
So we can write the decomposition of $\cF_{(1,1,\ldots,1)}$ as
\begin{equation}\label{eq:decompN11...1}
\cF_{(1,1,\ldots,1)}=\bigsqcup_{\varpi_n\in S_n^{|D|}}\bigsqcup_{m_1,\ldots,m_{n-1}\geq 0\atop\text{such that } \eqref{eq:constrain on degMj},\, \eqref{eq:stab cond-VHS}\text{ hold}}\cF_{(1,1,\ldots,1)}(\varpi_n,m_1,\ldots,m_{n-1}),
\end{equation}
and, from what we have done so far, the following is clear.

\begin{proposition} \label{prop:fixed-critical-11...1}
Let  $\varpi_n\in S_n^{|D|}$ as in \eqref{eq:permutations} and $m_1,\ldots,m_{n-1}$ non-negative integers verifying \eqref{eq:constrain on degMj} and \eqref{eq:stab cond-VHS} for every $l=2,\ldots,n$. Then the critical subvariety $\cF_{(1,1,\ldots,1)}(\varpi_n,m_1,\ldots,m_{n-1})$ is given by the pull-back diagram
  \[\displaystyle{\begin{CD}
    \cF_{(1,1,\ldots,1)}(\varpi_n,m_1,\ldots,m_{n-1})
      @>>> \Jac^{d_n}(X) \\
    @VVV @VVV \\
\prod_{j=1}^{n-1}\Sym^{m_j}(X) @>>> \Jac^{\sum_j jm_j}(X)\ ,
  \end{CD}}\]
  where:
  \begin{itemize}
  \item the top map is $(V,\vp)=(\bigoplus_jL_j, \sum_j\vp_j) \mapsto L_n$;
    \item $d_n=\frac{1}{n}\sum_{j=1}^{n-1}j(m_j+s_j(\varpi_n))-(n-1)(g-1+|D|/2)$;
  \item the vertical map on the left
  is given by $(V,\vp)=(\bigoplus_jL_j, \sum_j\vp_j) \mapsto  (\divisor(\vp_1),\ldots,\divisor(\vp_{n-1}))$;
  \item the map on the bottom is $(D_1,\ldots,D_{n-1}) \mapsto  \mathcal{O}_X(\sum_jjD_j)$;
  \item the vertical map on the right is $L_n \mapsto
  L_n^n\Lambda^{-1}K^{\frac{n(n-1)}{2}}\left(\frac{n(n-1)}{2}D-\sum_{j=1}^{n-1}jS_j(\varpi_n)\right)$.
    \end{itemize}
\end{proposition}

\subsection{The $E$-polynomial of the variant part}

The proof of the next result uses the description $\cF_{(1,1,\ldots,1)}(\varpi_n,m_1,\ldots,m_{n-1})$ given in Proposition \ref{prop:fixed-critical-11...1}. It can be found essentially in \cite[Theorem 7.6 (iv)]{hitchin:1987}, \cite[Proposition 3.11]{gothen:1994} and \cite[Proposition 10.1]{hausel-thaddeus:2003}. Again it is essential that $n$ is prime.

Recall that the group $\Gamma_n$ acts on $\cM$ by \eqref{eq:Gamma_n action}. This action clearly preserves each component $\cF_{(1,1,\ldots,1)}(\varpi_n,m_1,\ldots,m_{n-1})$ of $\cF_{(1,1,\ldots,1)}$.

\begin{proposition}\label{prop:variant111}
Let $n$ be prime. The variant part of the cohomology of $\cF_{(1,1,\ldots,1)}(\varpi_n,m_1,\ldots,m_{n-1})$ is non-trivial only in degree $m_1+\cdots+m_{n-1}$.
More precisely, 
\[H^*(\cF_{(1,1,\ldots,1)}(\varpi_n,m_1,\ldots,m_{n-1}),\C)^\var\cong\bigoplus_{\gamma\in\Gn\setminus\{e\}}\bigotimes_{j=1}^{n-1}\Lambda^{m_j}H^1(X,L_\gamma^j),\] where $H^1(X,L_\gamma^j)$ denotes twisted cohomology with values in the local system $L_\gamma^j$, and $L_\gamma$ is the flat line bundle corresponding to $\gamma$.
\end{proposition}

We are now in position to determine the left-hand side of
\eqref{eq:polynom0}. 

\begin{proposition}\label{prop:E-pol-variant-part}
For any $n$ prime, and any generic system of small weights, the following holds:
\begin{equation}\label{eq:poly variant1111}
(uv)^{\dim(\cM)/2}E(\cF_{(1,1,\ldots,1)})^\var=\frac{n^{2g}-1}{n}(n!)^{|D|}(uv)^{(n^2-1)(g-1)+|D|n(n-1)/2}((1-u)(1-v))^{(n-1)(g-1)}.
\end{equation}
\end{proposition}

\proof By \eqref{eq:decompN11...1},
\begin{equation}\label{eq:E-poly1}
E(\cF_{(1,1,\ldots,1)})^\var=\sum_{\varpi_n\in S_n^{|D|}}\sum_{m_1,\ldots,m_{n-1}\atop\text{such that } \eqref{eq:constrain on degMj},\, \eqref{eq:stab cond-VHS}\text{ hold}}E(\cF_{(1,1,\ldots,1)}(\varpi_n,m_1,\ldots,m_{n-1}))^\var,
\end{equation}
and then we must multiply it by the factor $(uv)^{\dim(\cM)/2}$.

For any non-trivial $\gamma\in\Gn$ and any $j$,
$\dim H^1(X,L_\gamma^j)=g-1$, i.e., $\dim
H^{1,0}(X,L_\gamma^j)=g-1$.
Thus by Proposition \ref{prop:variant111}, we find that
\[E(\cF_{(1,1,\ldots,1)}(\varpi_n,m_1,\ldots,m_{n-1}))^\var(u,v)=(n^{2g}-1)\prod_{j=1}^{n-1}\sum_{p+q=m_j\atop 0\leq p,q\leq g-1}(-1)^{p+q}\binom{g-1}{p}\binom{g-1}{q}u^pv^q.\]
We need to sum this expression over all $k$-tuples of permutations $\varpi_n$ and over all non-negative integers $m_j$ such that \eqref{eq:constrain on degMj} and \eqref{eq:stab cond-VHS} hold. 

Regarding the summation over the $m_j$, note that the right hand side is zero whenever there is an $m_j>2g-2$.

Since the system of the (generic) weights is small, the summand
$$\sum_{p\in
  D}\left(\sum_{i=1}^n(n-l+1)\alpha_i(p)-n\sum_{j=l}^n\alpha_{a_j(p)}(p)\right)$$
in \eqref{eq:stab cond-VHS} is very close to zero as well. From this, and using the fact that $s_j(\varpi_n)\leq |D|$
for all $j$, one shows that \eqref{eq:stab cond-VHS} holds for
$m_1=\dots=m_{n-1}=2g-2$, hence holds for any choice of $m_j$ between
$0$ and $2g-2$ for every $j$. Therefore we can sum over all
$0\leq m_1,\ldots,m_{n-1}\leq 2g-2$ subject to condition
\eqref{eq:constrain on degMj}. This is done by taking
$\xi =\exp(2\pi i/n)$ and since, for a given integer $\nu\in\Z$, the
sum $\sum_{l=0}^{n-1}\xi^{l\nu}$ equals $n$ if $\nu \equiv 0 \pmod n$
and zero otherwise, we have that $E(\cF_{(1,1,\ldots,1)})^\var(u,v)$
in \eqref{eq:E-poly1} equals, if $n\geq 3$,
\begin{equation*}
\begin{split}
&\hspace{.48cm}\frac{n^{2g}-1}{n}\sum_{\varpi_n\in S_n^{|D|}}\sum_{m_1,\ldots,m_{n-1}=0}^{2g-2}\sum_{l=0}^{n-1}\xi^{l\sum_{j=1}^{n-1}j(m_j+s_j(\varpi_n))}\prod_{j=1}^{n-1}\sum_{p+q=m_j\atop 0\leq p,q\leq g-1}(-1)^{p+q}\binom{g-1}{p}\binom{g-1}{q}u^pv^q\\
&=\frac{n^{2g}-1}{n}\sum_{\varpi_n\in S_n^{|D|}}\sum_{l=0}^{n-1}\xi^{l\sum_{j=1}^{n-1}js_j(\varpi_n)}\prod_{j=1}^{n-1}\sum_{m_j=0}^{2g-2}\sum_{p+q=m_j\atop 0\leq p,q\leq g-1}(-1)^{p+q}\binom{g-1}{p}\binom{g-1}{q}u^pv^q\xi^{jlm_j}\\
&=\frac{n^{2g}-1}{n}\sum_{\varpi_n\in S_n^{|D|}}\sum_{l=0}^{n-1}\xi^{l\sum_{j=1}^{n-1}js_j(\varpi_n)}\prod_{j=1}^{n-1}(1-\xi^{jl}u)^{g-1}(1-\xi^{jl}v)^{g-1}\\
&=\frac{n^{2g}-1}{n}(n!)^{|D|}((1-u)(1-v))^{(n-1)(g-1)}+\frac{n^{2g}-1}{n}\left(\frac{(1-u^n)(1-v^n)}{(1-u)(1-v)}\right)^{g-1}(n\cS(n,d)-n!)^{|D|},\\
\end{split}
\end{equation*}
where in the last equality we used the fact that $n$ is prime, and where
\[\cS(n,d)=\#\bigg\{\varpi_n(p)\in S_n\st \sum_{j=1}^{n-1}js_j(\varpi_n(p))\equiv 0\pmod n\bigg\}.\]
For $n=2$ we perform the precise same computation, except that we use the expression corresponding to $n=2$ in \eqref{eq:constrain on degMj}, yielding 
\begin{equation*}
\begin{split}
E(\cF_{(1,1)})^\var&=2^{|D|-1}(2^{2g}-1)(uv)^{3g-3+|D|}((1-u)(1-v))^{g-1}\\ &+\frac{2^{2g}-1}{2}((1+u)(1+v))^{g-1}(2\cS(2,d)-2)^{|D|},
\end{split}
\end{equation*} where 
$\cS(2,d)=\#\big\{\varpi_2(p)\in S_2\st |D|+s_1(\varpi_2(p))\equiv 0\pmod 2\big\}$.

It is clear that that the values of both $\cS(n,d)$ and $\cS(2,d)$ are independent of $p\in D$. It is also clear that $\cS(2,d)=1$. Actually by Lemma \ref{lemma:combinatorialweights} below, we have $\cS(n,d)=(n-1)!$, hence, for any $n\geq 2$ prime, $(uv)^{\dim(\cM)/2}E(\cF_{(1,1,\ldots,1)})^\var(u,v)$ equals
\[
\frac{n^{2g}-1}{n}(n!)^{|D|}(uv)^{(n^2-1)(g-1)+|D|n(n-1)/2}((1-u)(1-v))^{(n-1)(g-1)},
\]
completing the proof.
\endproof

The next lemma completes the proof of Proposition \ref{prop:E-pol-variant-part}.

\begin{lemma}\label{lemma:combinatorialweights}
For any $n\geq 2$ and $d$, $\cS(n,d)=(n-1)!$.
\end{lemma}
\proof
This is a purely combinatorial proof. Since the number $\cS(n,d)$ is obviously independent of $p\in D$, we will remove it from the notation. Any permutation $\varpi_n\in S_n$ is obtained from a permutation $\varpi_{n-1}\in S_{n-1}$ by inserting $n$ in the appropriate position. Conversely, any $\varpi_{n-1}=a_1\,a_2\ldots a_{n-1}\in S_{n-1}$ produces $n$ distinct permutations in $S_n$, by inserting $n$ in $\varpi_{n-1}$ in each one of the possible $j$ positions of $\varpi_{n-1}$, where $j\in\{0,\ldots,n-1\}$. Write $\varpi_{n-1}(j)$ for such permutation in $S_n$, so that 
\begin{equation*}
\begin{split}
\varpi_{n-1}(0)&=n\,a_1\ldots a_{n-1},\\
\varpi_{n-1}(j)&=a_1\ldots a_j\, n\, a_{j+1}\ldots a_{n-1}\quad 1\leq j\leq n-2,\\
\varpi_{n-1}(n-1)&=a_1\ldots a_{n-1}\,n.
\end{split}
\end{equation*}

Fix any $\varpi_{n-1}\in S_{n-1}$. Let $A$ be the ordered set of the indexes $i$ between $1$ and $n-2$ where $s_i(\varpi_{n-1})=1$. In other words,
\[A=\{i_1,\ldots,i_s\st i_\ell<i_{\ell+1},\, a_{i_\ell}>a_{i_\ell+1}\},\] for some $s\in\{0,\ldots,n-2\}$ (where $A=\emptyset\Leftrightarrow s=0$). Notice that we always have $i_\ell\geq \ell$.
Let \[\sigma=\sum_{i=1}^{n-2}is_i(\varpi_{n-1})=\sum_{i\in A}i=i_1+\cdots+i_s.\]
For each $j=0,\ldots,n-1$, let $\sigma_n(j)\in\Z_n$ be the class modulo $n$ of the difference \[\sum_{i=1}^{n-1}is_i(\varpi_{n-1}(j))-\sigma.\]

We claim that for any $\varpi_{n-1}\in S_{n-1}$, 
\begin{equation}\label{eq:claim}
\{\sigma_n(j)\st j\in\{0,\ldots,n-1\}\}=\{0,\ldots,n-1\}.
\end{equation}
Note that this proves that $\cS(n,d)$ is in bijection with $S_{n-1}$ for any $d$, hence proves the lemma.

To prove \eqref{eq:claim}, we shall explicitly give for each $k\in\{0,\ldots,n-1\}$, the corresponding $j\in\{0,\ldots,n-1\}$ such that $\sigma_{n-1}(j)=k$.
\begin{itemize}
\item If $k=0$, then we have obviously to take $j=n-1$. Indeed, $\sum_{i=1}^{n-1}is_i(\varpi_{n-1}(n-1))=\sigma$, hence $\sigma_n(n-1)=0$.

\item If $1\leq k\leq s$, take $j=i_\ell\in A$, where $\ell=s-k+1$. Then $\sum_{i=1}^{n-1}is_i(\varpi_{n-1}(i_\ell))=\sigma+s-\ell+1=\sigma+k$, so $\sigma_n(i_\ell)=k$

\item If $s+1\leq k\leq s+ i_1$, then take $j=k-s-1$. In fact, since $j<i_1$, we have that $\sum_{i=1}^{n-1}is_i(\varpi_{n-1}(j))=j+1+\sigma+s=\sigma+k$, i.e., $\sigma_n(k-s-1)=k$.

\item Suppose now that $i_\ell+s-\ell+2\leq k\leq i_{\ell+1}+s-\ell$, for some $\ell\in\{1,\ldots,s-1\}$. Note that these situations are possible only if $i_{\ell+1}\geq i_\ell+2$. By taking $j=k-s+\ell-1$, one checks that $\sum_{i=1}^{n-1}is_i(\varpi_{n-1}(j))=\sigma+j+s-\ell+1=\sigma+k$, and $\sigma_n(k-s+\ell-1)=k$.

\item Finally, if $i_s+2\leq k\leq n-1$, choose $j=k-1$. Indeed, $\sum_{i=1}^{n-1}is_i(\varpi_{n-1}(j))=\sigma+j+1=\sigma+k$. Hence, $\sigma_n(k-1)=k$.
\end{itemize}

In these items we ran through all the possible values of $k\in\{0,\ldots,n-1\}$, exactly once each, and we found a bijection with the positions $j\in\{0,\ldots, n-1\}$ such that $\sigma_n(j)=k\in\Z_n$.
This proves \eqref{eq:claim} and thus the lemma.
\endproof

The following corollary proves the last statement of Theorem \ref{thm:main2}.

\begin{corollary}\label{cor:E-pol-variant-part-n23}
If $n=2,3$, then \eqref{eq:poly variant1111} holds for any system of generic weights.
\end{corollary}
\proof From Proposition \ref{prop:E-pol-variant-part} we know that
\eqref{eq:poly variant1111} holds under the assumption of small
weights. But by Corollary \ref{cor:alpha-independent-E-poly} we know
that the $E$-polynomials $E(\cM)$ and $E(\cM)^{\Gn}$ are independent
of the (generic) weights, hence so is $E(\cM)^\var$ for any
$n$. Theorem \ref{thm:varpartn23} implies then that
$E(\cF_{(1,1,\ldots,1)})^\var$ is also independent of the generic
weights whenever $n=2,3$. Therefore the formula we reached is valid
for any generic weights, for such $n$.
\endproof

\section{Unramified cyclic covers, norm maps and Pryms}
\label{sec:unramified-Norm,Prym,actions}

The purpose of the following section is to recall some classical
results about Prym varieties of unramified coverings, essentially
going back to Narasimhan--Ramanan \cite{narasimhan-ramanan:1975} and
Mumford \cite{mumford:1971}, and corresponding to Section~7 of
\cite{hausel-thaddeus:2003}. For the benefit of the interested reader,
we have included complete proofs.

\subsection{Connected components of the kernel of a norm map}
In Section \ref{subsec:Hitmap}, we considered the Prym variety of a
ramified cover in the context of the Hitchin fibration.  In the
 case of an unramified cover the structure
of the kernel of the norm map turns out to be quite different.

Let $n$ be a prime number. Fix $\gamma\in \Gn$ and let $L_\gamma$ be
the corresponding $n$-torsion line bundle on $X$. Denote the
associated unramified regular $n$-cover by
\begin{equation}\label{eq:unram n-cover}
\pi:X_\gamma\to X.
\end{equation} 
Recall that $X_\gamma$ is the spectral cover of $X$ defined as the
curve in the total space $|L_\gamma|$ of $L_\gamma$ defined by the
equation $\lambda^n-1=0$, where
$\lambda\in H^0(|L_\gamma|,p^*L_\gamma)$ is the tautological section,
and $p:|L_\gamma|\to X$ is the projection. Then $\pi$ is the
restriction of $p$ to $X_\gamma$.  The line bundle $\pi^*L_\gamma$ is
trivial over $X_\gamma$ since the nowhere vanishing section
$\lambda:\mathcal O_{X_\gamma}\to\pi^*L_\gamma$ gives a canonical
trivialization.

Let $\Pic(X)$ be the Picard group of $X$ and $\Pic^i(X)$ be the component
corresponding to line bundles of degree $i$, so that
$\Pic^0(X)\cong\Jac(X)$, and
\[\Pic(X)=\bigsqcup_{i\in\Z}\Pic^i(X).\] 
Consider the same notations for the curve $X_\gamma$.  The dimension
of $\Pic(X)$ is $g$ while the dimension of $\Pic(X_\gamma)$ is the
genus of $X_\gamma$, given by $n(g-1)+1$.

The pullback map $\pi^*:\Pic(X)\to\Pic(X_\gamma)$ is not injective
neither surjective.  The non-surjectivity of $\pi^*$ is clear by
dimensional reasons and also because
$\pi^*(\Pic^i(X))\subset\Pic^{ni}(X_\gamma)$. The next proposition
provides the description of the image.  Consider the Galois group of
the covering $\pi:X_\gamma\to X$. It is isomorphic to $\Z_n$, which we
consider as the group of the $n$-th roots of unity. Let
$\xi=\exp(2\pi i/n)$ be the standard generator. The Galois group
$\Z_n$ acts on $\Pic^i(X_\gamma)$ by pullback and obviously a line
bundle over $X_\gamma$ is fixed by $\Z_n$ if and only if it is fixed
by $\xi$.

\begin{proposition}\label{prop:kernel-image of pi*}
The kernel of $\pi^*$ is the finite free abelian group generated by $L_\gamma$, that is $\ker(\pi^*)\cong\Z_n$.
The image coincides with $\Pic(X_\gamma)^{\Z_n}$, i.e. the fixed point subvariety of $\Pic(X_\gamma)$ under the Galois group. So $\pi^*$ yields an isomorphism $\Pic^i(X)/\Z_n\cong\Pic^{ni}(X_\gamma)^{\Z_n}$.
\end{proposition}
\proof
We already know that $\pi^*L^j_\gamma\cong\cO_{X_\gamma}$, for any $j=0,\ldots,n-1$. For the converse, take $L$ a degree $0$ line bundle on $X$ whose pullback is trivial. Then 
\[\cO_X\oplus L_\gamma^{-1}\oplus\cdots\oplus L_\gamma^{-(n-1)}\cong\pi_*\cO_{X_\gamma}\cong\pi_*\pi^*L\cong L\otimes(\cO_X\oplus L_\gamma^{-1}\oplus\cdots\oplus L_\gamma^{-(n-1)})\]
which implies that $L$ must be some power of $L_\gamma$. 

Regarding the image of $\pi^*$, since $\pi\circ\xi=\pi$, it is clear that $\xi^*\pi^*L=\pi^*L$ for any $L\in\Pic(X)$. Conversely, if $F\in\Pic(X_\gamma)$ is fixed by $\xi$, then $F$ descends to a line bundle $L$ in $X$ so that $F=\pi^*L$.
\endproof

We shall now consider the norm map associated to the unramified cover $\pi:X_\gamma\to X$. There are several incarnations of this map, all of them compatible with each other. We will consider three of them and use the same notation for all. The context will clarify the ones we are using. The norm map on divisors is given by
\[\Nm_\pi:\Div(X_\gamma)\to\Div(X),\quad E=\sum_pE(p)p\mapsto\Nm_\pi(E)=\sum_pE(p)\pi(p).\] 
(This has already been defined before, for more general coverings, in Definition \ref{def:NmDiv}.)
We also have the norm map on the fields of non-zero meromorphic functions, given by
\begin{equation}\label{eq:Normfuncfields}
\Nm_\pi:\cM(X_\gamma)^*\to\cM(X)^*,\quad \Nm_\pi(f)(p)=\prod_{q\in\pi^{-1}(p)}f(q).
\end{equation}
It is clear that $\Nm_\pi(\divisor(f))=\divisor(\Nm_\pi(f))$, for any $f\in\cM(X_\gamma)^*$, hence the norm map on divisors induces the norm map on the Picard groups, i.e., on line bundles:
\begin{equation}\label{eq:Normlinebundles}
\Nm_\pi:\Pic(X_\gamma)\to\Pic(X),\quad \cO_{X_\gamma}(E)\mapsto\cO_X(\Nm_\pi(E)).
\end{equation}

Let $\ker(\Nm_\pi)$ be the subvariety of $\Jac(X_\gamma)$ defined as the kernel of \eqref{eq:Normlinebundles}, and consider the group homomorphism
\begin{equation}\label{eq:projection p}
p:\Pic(X_\gamma)\to\ker(\Nm_\pi),\qquad L\mapsto L^{-1}\otimes \xi^*L.
\end{equation}
It is well-defined since $\Nm_\pi(L^{-1}\otimes \xi^*L)=\Nm_\pi(L^{-1})\otimes\Nm_\pi(\xi^* L)=\Nm_\pi(L)^{-1}\otimes\Nm_\pi(L)=\cO_X$.

The following is a generalization to $n\geq 2$ of Lemma 1 of Mumford \cite{mumford:1971}.

\begin{proposition}\label{prop:psurjective}
The homomorphism $p$ is surjective and the same holds for the restriction of $p$ to the disjoint union $\displaystyle\bigsqcup_{i=0}^{n-1}\Pic^i(X_\gamma)$.
\end{proposition}
\proof
Let $M\in\ker(\Nm_\pi)\subset\Jac(X_\gamma)$. Then $M$ must be isomorphic to $\cO_{X_\gamma}(F)$, for some degree $0$ divisor $F$, such that $\Nm_\pi(F)=\divisor(f)$, for some non-zero meromorphic function $f$ on $X$. But the norm map  \eqref{eq:Normfuncfields} on function fields is surjective (see \cite{lang:1952} and also \cite[p. 282]{arbarello-cornalba-griffiths-harris:1985}), hence $f=\Nm_\pi(g)$ for some $g\in\cM(X_\gamma)^*$. Let $G=\divisor(g)$.
Then $\Nm_\pi(G)=\Nm_\pi(\divisor(g))=\divisor(\Nm_\pi(g))=\divisor(f)=\Nm_\pi(F)$. 
Define $\bar F=F-G$. Then $\Nm_\pi(\bar F)=0$, hence $\bar F$ must be of the form
$\bar F=\sum_{p\in X_\gamma}\bar F(p)p$
with
\begin{equation}\label{eq:sum elements divisor =0}
\sum_{i=0}^{n-1}\bar F(\xi^i(p))=0.
\end{equation}

Now choose one and only one element in the support of $\bar F$ in each fibre of $\Nm_\pi$ over the support of $\Nm_\pi(\bar F)$. This yields a collection of points $p_1,\ldots,p_m$ in the support of $\bar F$ such that $\pi(p_i)\neq\pi(p_j)$, for $i\neq j$.  
For each $l=1,\ldots,m$, choose integers $k_1(l),\ldots,k_n(l)$ such that 
\begin{equation}\label{eq:relation-new-integers}
\bar F(\xi^i(p_l))=-k_{i+1}(l)+k_i(l)
\end{equation} for every $i=0,\ldots,n$ and where, by definition, $k_0(l)=k_n(l)$. This is possible (in an infinite number of ways) due to \eqref{eq:sum elements divisor =0}.
Define the divisor \[\tilde F=\sum_{l=1}^m\sum_{i=0}^{n-1}k_{i+1}(l)\xi^i(p_l)\] and the corresponding line bundle $\tilde L\cong\cO_{X_\gamma}(\tilde F)$. It follows from \eqref{eq:relation-new-integers} that 
$\cO_{X_\gamma}(\bar F)\cong\tilde L^{-1}\otimes\xi^*\tilde L$, that is, \[M\cong\tilde L^{-1}\otimes\xi^*\tilde L=p(\tilde L),\] because $M\cong\cO_{X_\gamma}(F)$, $\bar F=F-G$ and $\cO_{X_\gamma}(G)$ is trivial. Hence $p$ is surjective.

To show that its restriction to $\bigsqcup_{i=0}^{n-1}\Pic^i(X_\gamma)$ is also surjective, consider the same line bundle $M=p(\tilde L)$ and let $\tilde d$ be the degree of $\tilde L$.  Let $a\in\{0,\ldots,n-1\}$ be the reduction of $\tilde d$ modulo $n$ and choose a line bundle $M'$ on $X$, of degree $(\tilde d-a)/n$. Then $\deg(\tilde L\otimes \pi^* M')=a$ and 
\[M=\tilde L^{-1}\otimes \pi^*M^{\prime {-1}}\otimes\xi^*(\tilde L\otimes\pi^*M')=p(\tilde L\otimes\pi^*M'),\] completing the proof.
 \endproof

\begin{proposition}
The kernel of $p$ equals the image of $\pi^*$. Hence \[\ker(p)\cap\bigsqcup_{i=0}^{n-1}\Pic^i(X_\gamma)=\pi^*(\Pic^0(X))\cong\Pic^0(X)/\Z_n.\]
\end{proposition}
\proof
The kernel of $p$ is precisely given by the fixed points under $\xi$ (thus under $\Z_n$), hence the result follows immediately from Proposition \ref{prop:kernel-image of pi*}. The second part follows because $\pi^*(\Pic^i(X))\subset\Pic^{ni}(X_\gamma)$.
\endproof

The previous propositions can be summarized in the next corollary:

\begin{corollary}
The following sequence of groups is exact: \[0\to\Z_n\to\Pic(X)\xrightarrow{\pi^*}\Pic(X_\gamma)\xrightarrow{p}\ker(\Nm_\pi)\to 0.\]
Moreover, the restriction of $p$ to $\bigsqcup_{i=0}^{n-1}\Pic^i(X_\gamma)$, 
\begin{equation}\label{eq:projection p-restricted}
p:\bigsqcup_{i=0}^{n-1}\Pic^i(X_\gamma)\to\ker(\Nm_\pi)
\end{equation} is a holomorphic $\Pic^0(X)/\Z_n$-principal bundle.
\end{corollary}

The following is now immediate from the stated property of the map \eqref{eq:projection p-restricted}. 
\begin{corollary}\label{cor:connectedcompKerNm}
The kernel $\ker(\Nm_\pi)$ of the norm map \eqref{eq:Normlinebundles} has $n$ connected components, which are labeled by the $n$ connected components of $\bigsqcup_{i=0}^{n-1}\Pic^i(X_\gamma)$ via the group homomorphism \eqref{eq:projection p}.
\end{corollary}

Recall from Definition \ref{def:Prym} that the Prym variety of $X_\gamma$ associated to the covering $\pi:X_\gamma\to X$ is the abelian variety defined as the connected component of $\ker(\Nm_\pi)$ containing the identity:
\[\Prym_\pi(X_\gamma)=\ker(\Nm_\pi)_0.\]
Note that now we do not have the equality corresponding to \eqref{eq:Prym=ker-ramified}.

\begin{proposition}\label{prop:connectedcompKerNm-deg}
Two line bundles $M_1$ and $M_2$ are in the same connected component of the kernel of $\Nm_\pi$ if and only if $M_1=L_1^{-1}\otimes\xi^*L_1$ and $M_2=L_2^{-1}\otimes\xi^*L_2$, with $\deg(L_1)=\deg(L_2)$. In particular $\Prym_\pi(X_\gamma)$ is the subspace of $\ker(\Nm_\pi)$ consisting of those line bundles of the form $L^{-1}\otimes\xi^*L$, with $\deg(L)=0$.
\end{proposition}

Thus the following sequence of groups is exact:
\[
0\to\Z_n\to\Pic^0(X)\xrightarrow{\pi^*}\Pic^0(X_\gamma)\xrightarrow{p}\Prym_\pi(X_\gamma)\to 0,
\]
hence 
\begin{equation}\label{eq:dimPrym}
\dim(\Prym_\pi(X_\gamma))=(n-1)(g-1).
\end{equation} 

We shall need a generalization of Proposition \ref{prop:connectedcompKerNm-deg} to any fibre of the norm map, and not only its kernel. That is easily achieved since such fibre is a torsor for the kernel, hence being isomorphic to $\ker(\Nm_\pi)$ although not canonically. Let then $\Lambda$ be a degree $d$, holomorphic line bundle over $X$. Choose an arbitrary line bundle $L_0\in\Nm_\pi^{-1}(\Lambda)$. Given this choice, we have the obvious isomorphism
\[
\ker(\Nm_\pi)\xrightarrow{\cong}\Nm_\pi^{-1}(\Lambda),\qquad M\mapsto M\otimes L_0.
\]
Consider the union $\bigsqcup_{i=d}^{d+n-1}\Pic^i(X_\gamma)$. The same kind of isomorphism holds,
\[
\bigsqcup_{i=0}^{n-1}\Pic^i(X_\gamma)\xrightarrow{\cong}\bigsqcup_{i=d}^{d+n-1}\Pic^i(X_\gamma),\qquad L\mapsto L\otimes L_0
\]
and we have the analogue of the restriction of the map $p$ to $\bigsqcup_{i=0}^{n-1}\Pic^i(X_\gamma)$,
\begin{equation}\label{eq:pL0}
p_{L_0}:\bigsqcup_{i=d}^{d+n-1}\Pic^i(X_\gamma)\to\Nm_\pi^{-1}(\Lambda),\qquad p_{L_0}(L)=L^{-1}\otimes\xi^*L\otimes L_0.
\end{equation}
Hence the following diagram commutes:
\[\displaystyle{\begin{CD}
    \displaystyle\bigsqcup_{i=0}^{n-1}\Pic^i(X_\gamma)
      @>\cong>> \displaystyle\bigsqcup_{i=d}^{d+n-1}\Pic^i(X_\gamma) \\
    @VpVV @VVp_{L_0}V \\
\ker(\Nm_\pi) @>>\cong> \Nm_\pi^{-1}(\Lambda)\ ,
  \end{CD}}\]

The twisted version of the previous results reads then as follows. This result (and the particular case of Corollary \ref{cor:connectedcompKerNm}) goes back at least to Narasimhan--Ramanan \cite{narasimhan-ramanan:1975}.

\begin{proposition}\label{prop:connectedcompKerNm-deg=d}
The map $p_{L_0}$ is a holomorphic $\Pic^0(X)/\Z_n$-principal bundle and the $n$ connected components of $\Nm_\pi^{-1}(\Lambda)$ are labeled by the degree $i\in\{d,\ldots,d+n-1\}$. Moreover, $\Pic^i(X)$ is a holomorphic $\Pic^0(X)/\Z_n$-principal bundle over a connected component of $\Nm_\pi^{-1}(\Lambda)$.
\end{proposition}

Thus each connected component of $\Nm_\pi^{-1}(\Lambda)$ is a torsor for $\Prym_\pi(X_\gamma)$.

\subsection{The action of the Galois group}

We now wish to see how the Galois group $\Z_n$ acts on the on
components of the fibre of the norm map. This is not strictly
necessary for what follows but we include it for completeness.

We continue with an unramified $n$-cover (\ref{eq:unram n-cover}) and
a line bundle $\Lambda$ over $X$ of degree $d$. Let
$\pi_0(\Nm_\pi^{-1}(\Lambda))$ be the set consisting of the $n$ connected
components of the fibre of the norm map of over $\Lambda$. Let $(n,d)$
denote the greatest common divisor of $n$ and $d$.

\begin{proposition}
The $\Z_n$-orbit of any element of $\pi_0(\Nm_\pi^{-1}(\Lambda))$ has $n/(n,d)$ elements. In particular, $\Z_n$ acts trivially on $\pi_0(\Nm_\pi^{-1}(\Lambda))$ if and only if $d$ is a multiple of $n$ and acts transitively if and only if $n$ and $d$ are coprime.
\end{proposition}
\proof
Let $M$ be a line bundle of degree $d$ such that $\Nm_\pi(M)=\Lambda$. Then $M=p_{L_0}(L)$ for some degree $i\in\{d,\ldots,d+n-1\}$ line bundle $L$ over $X_\gamma$, where $p_{L_0}$ is defined in \eqref{eq:pL0},
\[M=L^{-1}\otimes\xi^*L\otimes L_0.\]
By Proposition \ref{prop:connectedcompKerNm-deg=d}, the component of $\Nm_\pi^{-1}(\Lambda)$ where $M$ lies is determined by the degree $i$ of $L$. Since 
\[\xi^*M=\xi^*L^{-1}\otimes\xi^*\xi^*L\otimes\xi^*L_0=(\xi^*L\otimes L_0)^{-1}\otimes\xi^*(\xi^*L\otimes L_0)\otimes L_0=p_{L_0}(\xi^*L\otimes L_0)\]
then $\xi^*M$ lies in the component determined by the degree of $\xi^*L\otimes L_0$, which is $i+d$. So $\xi^*M$ is in the same connected component as $M$ is and only if $i+d$ is equal to $i$ modulo $n$, that is $d$ is a multiple of $n$. 

 If $n$ does not divide $d$ then, from what we saw, the orbit of $M$ on $\pi_0(\Nm_\pi^{-1}(\Lambda))$ is determined by class in $\Z_n$ of the numbers $i+jd$, with $j=0,\ldots, n-1$. Hence we conclude that the orbit of $M$ runs over $n/(n,d)$ different connected components of $\Nm_\pi^{-1}(\Lambda)$. 
\endproof

We know that $\Pic(X_\gamma)^{\Z_n}$ is the image of $\pi^*$.
Let us now see what is its intersection with a fibre of $\Nm_\pi$.

\begin{proposition}
The intersection $\Pic(X_\gamma)^{\Z_n}\cap\Nm_\pi^{-1}(\Lambda)$ is the image of $\pi^*|_{\Gn}$. In particular, it is empty if $d$ is not a multiple of $n$. If not empty, it has $n^{2g-1}$ elements.
\end{proposition}
\proof
Since $\Z_n$ is generated by $\xi$, it is enough to consider a line bundle $L\in\Nm_\pi^{-1}(\Lambda)$ such that $L\cong\xi^*L$. By Proposition \ref{prop:kernel-image of pi*}, we have that $L\cong\pi^*N$ for some line bundle $N$ over $X$. Then $\Lambda\cong\Nm_\pi(\pi^*N)\cong N^n$. This is only possible if $d=n\deg(N)$, showing that otherwise the intersection is empty.

This proves that the intersection is the image, under $\pi^*$, of the $n^{2g}$ $n$th-roots of $\Lambda$. But this image only has $n^{2g-1}$ elements since two such roots are pulled-back to the same element whenever they differ by a power of $L_\gamma$.
\endproof

\subsection{The action of $\Gn$ and the Weil pairing}

Now we consider the action of $\Gn$ on $\Nm_\pi^{-1}(\Lambda)$. An element $\delta\in\Gn$ acts on $M\in\Nm_\pi^{-1}(\Lambda)$ by 
\begin{equation}\label{eq:actGnkerNm}
M\mapsto M\otimes \pi^*L_\delta,
\end{equation} where $L_\delta$ is the $n$-torsion line bundle on $X$ corresponding to $\delta$. Indeed,  $\Nm_\pi(\pi^*L_\delta)=L_\delta^n\cong\cO_X$, therefore $\Nm_\pi(M\otimes\pi^*L_\delta)=\Lambda$.

\begin{proposition}
The subgroup $\Z_n=\langle\gamma\rangle\subset\Gn$ acts trivially  on $\Nm_\pi^{-1}(\Lambda)$, and the $\Gn$-action induces a free action of $\Gn/\Z_n\cong\Z_n^{2g-1}$ on $\Nm_\pi^{-1}(\Lambda)$.
\end{proposition}
\proof
The elements $\delta\in\Gn$ which fix some element in $\Nm_\pi^{-1}(\Lambda)$ are the ones  such that $\pi^*L_\delta\cong\cO_{X_\gamma}$, and thus they fix every point in $\Nm_\pi^{-1}(\Lambda)$. By Proposition \ref{prop:kernel-image of pi*}, $\delta$ is such an element if and only if $L_\delta$ is a power of $L_\gamma$, i.e., $\delta\in\langle\gamma\rangle\cong\Z_n$. So the $\Gn$-action factors through a free $\Gn/\Z_n\cong\Z_n^{2g-1}$-action on $\Nm_\pi^{-1}(\Lambda)$.
\endproof

We must especially study the action of $\Gn$ on the set of connected components of $\Nm_\pi^{-1}(\Lambda)$. Since $\Gn\cong H^1(X,\Z_n)$, we have a pairing on $\Gn$ given by cup product followed by evaluation on the fundamental class:
\begin{equation}\label{eq:Weilpairing}
\langle\, ,\rangle:\Gn\times\Gn\to\Z_n,
\end{equation} where $\Z_n$ is given the multiplicative structure.
This is a symplectic pairing, called the \emph{Weil pairing}. 

It will be convenient to give a different (but equivalent) definition of the Weil pairing. First, given a meromorphic function $f$ on $X$ and a divisor $D$ on $X$ whose support is disjoint from the support of the divisor of $f$, define 
\begin{equation}\label{eq:f(D)}
f(D)=\prod_{p\in X}f(p)^{D(p)}.
\end{equation}
Weil reciprocity (see for instance \cite[p. 283]{arbarello-cornalba-griffiths-harris:1985}) states that $f(\divisor(g))=g(\divisor(f))$ for any pair of meromorphic functions $f,g$ on $X$.
Using this, the Weil pairing \eqref{eq:Weilpairing} can also be defined as follows. Take two $n$-torsion line bundles $L_1,L_2$ on $X$. Let $D_1$ and $D_2$ be divisors, with disjoint support, so that $L_i\cong\cO_X(D_i)$. Then $nD_i=\divisor(f_i)$ for some meromorphic function $f_i:X\to\C$, and
\begin{equation}\label{eq:Weilpairing2}
\langle L_1,L_2\rangle=\frac{f_1(D_2)}{f_2(D_1)}\in\Z_n.
\end{equation}

\begin{remark}
Sometimes (cf. \cite{arbarello-cornalba-griffiths-harris:1985}), $\langle L_1,L_2\rangle$ is defined as $\frac{n}{2\pi i}\log\frac{f_1(D_2)}{f_2(D_1)}$, but this is by considering $\Z_n$ with the additive structure.
\end{remark}

Recall that $\xi=\exp(2\pi i/n)\in\Z_n$ denotes the standard generator of the multiplicative group $\Z_n$.
Recall also that if $L_\delta$ is a $n$-torsion line bundle on $X$, then $\pi^*L_\delta$ lies in the kernel of the norm map, so by Proposition \ref{prop:psurjective} it is of the form $F_\delta^{-1}\otimes\xi^*F_\delta$ for some line bundle $F_\delta$ on $X_\gamma$ of degree between $0$ and $n-1$. The next result generalizes Mumford \cite[Lemma 2]{mumford:1971} to $n\geq 2$.

\begin{proposition}\label{prop:Weilpairing-degree}
Let $L_\gamma$ be the line bundle corresponding to $\gamma\in\Gn\setminus\{e\}$ and let $L_\delta$ be any $n$-torsion line bundle on $X$ which is not a power of $L_\gamma$. Let $F_\delta$ be a line bundle on $X_\gamma$, of degree between $0$ and $n-1$, such that $\pi^*L_\delta\cong F_\delta^{-1}\otimes\xi^*F_\delta$. Then there exists a non-zero integer $l(\gamma)$ between $1$ and $n-1$, depending only on $\gamma$, such that
\[\langle L_\delta,L_\gamma\rangle=\xi^{l(\gamma)\deg(F_\delta)}.\]
\end{proposition}
\proof
It will be convenient to use definition \eqref{eq:Weilpairing2}. Let $D_\gamma$ and $D_\delta$ be divisors on $X$, with disjoint support, such that $L_\gamma\cong\cO_X(D_\gamma)$ and $L_\delta\cong\cO_X(D_\delta)$ and such that $nD_\gamma=\divisor(f_\gamma)$ and $nD_\delta=\divisor(f_\delta)$ for some non-zero meromorphic functions $f_\gamma,f_\delta$.
Since $\pi^*L_\gamma$ is trivial, then $\pi^*D_\gamma=\divisor(g)$ for some meromorphic function $g\in\cM(X_\gamma)^*$. From this it follows that $f_\gamma=\Nm_\pi(g)$.
On the other hand, if $F_\delta\cong\cO_{X_\gamma}(D)$ for some divisor $D$ on $X_\gamma$, then there is a non-zero meromorphic function $h$ on $X_\gamma$ such that $\pi^*D_\delta=-D+\xi D+\divisor(h)$. It also follows that $f_\delta=\Nm_\pi(h)$.

Since $\divisor(\xi^*g)=\xi^*\pi^*D_\gamma=\pi^*D_\gamma=\divisor(g)$, there is a non-zero complex number $\lambda(\gamma)$, depending only on $g$, i.e. only on $\gamma$, such that 
\begin{equation}\label{eq:equivariance of g}
\xi^*g=\lambda(\gamma)g.
\end{equation} Note that $\lambda(\gamma)\neq 1$, since otherwise that would be saying that $L_\gamma$ was trivial. Furthermore,
$g^n=\pi^*f_\gamma=\xi^*\pi^*f_\gamma=(\xi^*g)^n$, hence $\lambda(\gamma)^n=1$, thus $\lambda(\gamma)=\xi^{l(\gamma)}$ for some integer $l(\gamma)\in\{0,\ldots,n-1\}$ depending only on $\gamma$.

Recall that $\Nm_\pi(g)(p)=\prod_{\tilde p\in\pi^{-1}(p)}g(\tilde p)$. Then it is easy to see directly from the definition \eqref{eq:f(D)}, that $\Nm_\pi(g)(D_\delta)=g(\pi^*D_\delta)$. Similarly for $\Nm_\pi(h)(D_\gamma)$.
Then, using \eqref{eq:equivariance of g}, we get
\begin{equation*}
\begin{split}
\langle L_\delta,L_\gamma\rangle&=\frac{\Nm_\pi(h)(D_\gamma)}{\Nm_\pi(g)(D_\delta)}=\frac{h(\pi^*D_\gamma)}{g(\pi^*D_\delta)}\\
&=\frac{h(\divisor(g))}{g(\divisor(h)-D+\xi D)}=\frac{g(D)}{g(\xi D)}\\
&=\frac{\prod_{p\in X_\gamma}g(p)^{D(p)}}{\prod_{p\in X_\gamma}g(p)^{(\xi D)(p)}}=\frac{\prod_{p\in X_\gamma}g(p)^{D(p)}}{\prod_{p\in X_\gamma}(\lambda(\gamma)^{-1}g(\xi\cdot p))^{D(\xi\cdot p)}}\\
&=\lambda(\gamma)^{\sum_{p\in X_\gamma}D(\xi\cdot p)}=\lambda(\gamma)^{\deg(F_\delta)}\\
&=\xi^{l(\gamma)\deg(F_\delta)}
\end{split}
\end{equation*}
as claimed.
\endproof

If $n=2$ we have $l(\gamma)=1$ for every non-trivial $\gamma$, recovering Mumford's result.

The Weil pairing corresponds to intersection form in homology $H_1(X,\Z_n)$.
From this definition it is clear that, for each $\gamma\in\Gn\setminus\{e\}$, it is possible to choose a basis 
\begin{equation}\label{eq:basis}
(\gamma,\delta_0,\delta_1,\ldots,\delta_{2g-2})
\end{equation} of $\Gn$ including $\gamma$, and such that 
\begin{equation}\label{eq:Wpairinchosenbasis}
\langle L_{\delta_i},L_\gamma\rangle=\begin{cases}
\xi^{l(\gamma)}&\, i=0\\
1&\, i\neq 0.
\end{cases}
\end{equation}
We assume from now on that such a basis has been chosen.

Now we describe the $\Gn$-action on the $n$ connected components of $\Nm_\pi^{-1}(\Lambda)$.

\begin{proposition}\label{prop:Gamma-action on components of Nm^-1}
Every element of $\Gn$ which is not in the subgroup $\langle\delta_0\rangle$ generated by $\delta_0$ acts trivially on $\pi_0(\Nm_\pi^{-1}(\Lambda))$, and $\langle\delta_0\rangle\cong\Z_n$ acts freely and transitively on $\pi_0(\Nm_\pi^{-1}(\Lambda))$. In particular, $\Gn$ acts transitively on $\pi_0(\Nm_\pi^{-1}(\Lambda))$.
\end{proposition}
\proof
We will use Proposition \ref{prop:connectedcompKerNm-deg=d}. Let $M\in\Nm_\pi^{-1}(\Lambda)$, so that $M=p_{L_0}(L)$ for some degree $i\in\{d,\ldots,d+n-1\}$ line bundle $L$ over $X_\gamma$, i.e.
\[M=L^{-1}\otimes\xi^*L\otimes L_0.\]
Since $\pi^*L_\delta\in\ker(\Nm_\pi)$ then, by Proposition \ref{prop:psurjective}, $\pi^*L_\delta=p(F_\delta)=F_\delta^{-1}\otimes\xi^*F_\delta$ for some line bundle $F_\delta$ on $X_\gamma$ of degree between $0$ and $n-1$. So
\[M\otimes\pi^*L_\delta=(L\otimes F_\delta)^{-1}\otimes\xi^*(L\otimes F_\delta)\otimes L_0.\]
If $\delta\not\in\langle\delta_0\rangle$, then by \eqref{eq:Wpairinchosenbasis}, $\langle L_\delta,L_\gamma\rangle=1$ hence, by Proposition \ref{prop:Weilpairing-degree},  $\deg(F_\delta)=0$, so that $\deg(L\otimes F_\delta)=i$ and therefore $M\otimes\pi^*L_\delta$ is in the same connected component as $M$ by Proposition \ref{prop:connectedcompKerNm-deg=d}. For $\delta_0$, for the same reasons we have $\deg(F_\delta)=1$, so $\deg(L\otimes F_\delta)=i+1$, thus the powers of $\delta_0$ act transitively on $\pi_0(\Nm_\pi^{-1}(\Lambda))$.
\endproof


\section{The polynomial $E_{\mathrm{st}}(\cM/\Gn)-E(\cM)^{\Gn}$}\label{sec:stringy E-pol}


\subsection{The subvarieties $\cM^\gamma$}

The stringy $E$-polynomial of $\cM/\Gn$ is defined by \eqref{eq:stringyEpolynomial}, but the statement of Theorem \ref{thm:main2} is only about
\[
\sum_{\gamma\neq e}E(\cM^\gamma)^{\Gn}(uv)^{F(\gamma)}.
\]
This is the polynomial we aim to compute in the present section.  This requires the study of
the subvarieties $\cM^\gamma$ of fixed points under the action of each nontrivial
$\gamma\in\Gn$. This study is an adaptation
to parabolic Higgs bundles of the corresponding result for vector
bundles studied by Narasimhan and Ramanan in
\cite{narasimhan-ramanan:1975} (cf.\ Hausel and Thaddeus
\cite[Sec.~7]{hausel-thaddeus:2003} for the case of Higgs bundles in
the non-parabolic situation).

Recall that to each non-trivial $\gamma\in\Gn$ we associate an unramified cyclic $n$-cover $\pi:X_\gamma\to X$ as in \eqref{eq:unram n-cover}, with Galois group isomorphic to $\Z_n$ and $\Nm_\pi:\Pic^d(X_\gamma)\to\Pic^d(X)$ is the corresponding norm map. 

Let $D_\gamma=\pi^{-1}(D)$ be the inverse image in $X_\gamma$ of our fixed divisor $D=p_1+\cdots+p_{|D|}$ in $X$, from which we have our fixed generic parabolic type $\bm\a$,
\[0\leq \a_1(p)<\cdots<\a_n(p)<1,\]
for each $p\in D$. It is important to note that our genericity assumption on the parabolic type $\bm\a$ still holds, but that we do not need them to be small as it was required to prove in Proposition \ref{prop:E-pol-variant-part} the equality corresponding to the left-hand side of Theorem \ref{thm:main2}. 

Given
\begin{equation}\label{eq:permutation-weights fibers}
\varpi_n=(\varpi_n(p_1),\ldots,\varpi_n(p_{|D|}))\in S_n^{|D|},
\end{equation} 
we naturally construct a parabolic type of rank $1$, denoted by $\bm\a_\gamma(\varpi_n)$ on $D_\gamma$ as follows. 
For each $p\in D$, write \[\varpi_n(p)=a_1(p)\ a_2(p) \dots a_n(p)\in S_n.\] 
Then attach the weights $0\leq \a_1(p)< \cdots <\a_n(p) < 1$ to the set $\pi^{-1}(p)=\{q_1,\ldots,q_n\}\subset D_\gamma$
so that the point $q_i\in \pi^{-1}(p)$ is given the weight $\a_{a_i(p)}(p)$. This yields the parabolic type $\bm\a_\gamma(\varpi_n)$,
\[0\leq \a_{a_i(p)}(p)<1,\]
at each $q_i\in \pi^{-1}(p)\subset D_\gamma$. Strictly speaking, this depends on a choice of an ordering of the points in $\pi^{-1}(p)$ for each $p$. This ordering was implicitly chosen when we wrote $\pi^{-1}(p)=\{q_1,\ldots, q_n\}$. So, without the choice of that ordering, $\varpi_n$ in \eqref{eq:permutation-weights fibers} belongs to a torsor for the group $S_n^{|D|}$. In any case, any ordering is valid for our purposes.

For each $\varpi_n\in S_n^{|D|}$, consider the moduli space $\cM_d^{\bm\a_\gamma(\varpi_n)}(\C^*)$ of strongly parabolic Higgs line bundles (i.e., $\C^*$-Higgs bundles) of degree $d\in\Z$, over $X_\gamma$, of parabolic type $\bm\a_\gamma(\varpi_n)$ over $D_\gamma$.
Let $K_\gamma=\pi^*K$ be the canonical line bundle of $X_\gamma$.

\begin{lemma}\mbox{}
\begin{enumerate}
\item
For every $\varpi_n\in S_n^{|D|}$, the moduli space $\cM_d^{\bm\a_\gamma(\varpi_n)}(\C^*)$ is isomorphic to the moduli space of Higgs line bundles of degree $d$ over $X_\gamma$, i.e., to the cotangent bundle $T^*\Pic^d(X_\gamma)\cong\Pic^d(X_\gamma)\times H^0(X_\gamma,K_\gamma)$.
\item The disjoint union $\bigsqcup_{\varpi_n\in S_n^{|D|}}\cM_d^{\bm\a_\gamma(\varpi_n)}(\C^*)$ is isomorphic to $T^*\Pic^d(X_\gamma)\times S_n^{|D|}$.
\end{enumerate}
\end{lemma}
\proof
Take a strongly parabolic Higgs line bundle $(F,\phi)\in\cM_d^{\bm\a_\gamma(\varpi_n)}(\C^*)$. Then $F\in\Pic^d(X_\gamma)$ and $\phi\in H^0(X_\gamma,K_\gamma(D_\gamma))$. But the fact that $\phi$ is strongly parabolic and $F$ is a line bundle implies that actually $\phi\in H^0(X,K_\gamma)$, thus the map that forgets the parabolic structure yields the isomorphism stated in (1). Then (2) follows from (1).
\endproof

\begin{proposition}
  \label{prop:pushforward-semistable}
Let $\bm\alpha=(\a_1(p),\ldots,\a_n(p))_{p\in D}$ be any generic parabolic type and let $\varpi_n\in S_n^{|D|}$. If $(F,\phi)\in\cM_d^{\bm\a_\gamma(\varpi_n)}(\C^*)$, then $(\pi_*F,\pi_*\phi)$ is a semistable strongly parabolic $\GL(n,\C)$-Higgs bundle of degree $d$ and parabolic type $\bm\a$.
\end{proposition}
\proof
Let $V=\pi_*F$ and $\vp=\pi_*\phi:V\to V\otimes K$. Since $\pi$ is unramified and $F$ has degree $d$, then so has the rank $n$ vector bundle $V$. Note that $K$ is a subsheaf of $K(D)$ so $\vp$ is also a section of $\End(V)\otimes K(D)$.
Let us see that $V$ has a parabolic structure at $D$ of type $\bm\a$. Let $p\in D$ and $\{q_1,\ldots,q_n\}=\pi^{-1}(p)\subset D_\gamma$.
The filtration of $V_p$ is given as
\begin{equation}\label{eq:filt-perm1}
V_p=V_{p,1}=\bigoplus_{i=1}^n F_{q_i}\supsetneq V_{p,2} \supsetneq \cdots \supsetneq V_{p,n}\supsetneq \{0\},\qquad 0\leq \a_1(p)< \cdots <\a_n(p) < 1,
\end{equation}
where, for $2\leq j\leq n$, 
\begin{equation}\label{eq:filt-perm2}
V_{p,j}=V_{p,j-1}/F_{q'}
\end{equation} with $q'$ being the point $q_i$ attached with the weight $\a_{j-1}(p)$ i.e. $q'=q_i$ such that $a_i(p)=j-1$. So $V_{p,n}=F_{q_i}$ such that $a_i(p)=n$, $F_{p,n-1}=F_{q_j}\oplus F_{q_i}$ with $j$ such that $a_j=n-1$, and so on.
Doing this for every point of $D$ determines the parabolic structure of $V$ of type $\bm\a$.

Notice that, conversely, the parabolic structure of $V$, given at each point of $D$ by \eqref{eq:filt-perm1} and \eqref{eq:filt-perm2} determines the element \eqref{eq:permutation-weights fibers} of $S_n^{|D|}$.

Over each $p\in D$, $\vp_p=\vp|_{V_p}$ is diagonal with respect to the decomposition $V_p=\bigoplus F_{q_i}$ of $V_p$. Look at $\phi\in H^0(X_\gamma,K_\gamma)$ as a section of $K_\gamma(D_\gamma)$ which vanishes at every $q_i\in D_\gamma$. Then it is clear that $\vp$ is strongly parabolic with respect to \eqref{eq:filt-perm1}. 
So $(V,\vp)$ is a strongly parabolic Higgs bundle of rank $n$, degree
$d$ and parabolic type $\bm\a$.  

It remains to check
semistability. For that, recall that $\xi=\exp(2\pi i/n)$ denotes the
standard generator of the Galois group $\Z_n$, and note that
$(\pi^*V,\pi^*\vp)\cong (F\oplus \xi^*F\oplus\cdots \oplus \xi^{\ast
  n-1}F,\phi\oplus \xi^*\phi\oplus\cdots \oplus \xi^{\ast n-1}\phi)$
is a strongly parabolic $\GL(n,\C)$-Higgs bundle on
$X_\gamma$ with parabolic structure over $D_\gamma$ induced from
\eqref{eq:filt-perm1}. Since
\[\pmu(\pi^*\pi_*F)=\frac{nd+\sum_{q\in
      D_\gamma}\sum_{i=1}^n\alpha_i(\pi(q))}{n}=d+\sum_{p\in
    D}\sum_{i=1}^n\alpha_i(p)=\pmu(\xi^{j\ast}F),\] 
for every $j=0,\ldots,n-1$, then $\pi^*V\cong \bigoplus_{j=1}^n\xi^{j\ast}F$ is a direct sum of  parabolic line bundles, all of the same slope, and therefore it is semistable \cite[Corollaire 10, p. 71]{seshadri:1982}. 
Take a $\vp$-invariant subbundle $V'\subset V$ of degree $d'$ and rank $n'$. Then $\pi^*V'$ is a $\pi^*\vp$-invariant subbundle of $\pi^*V$. By semistability of $\pi^*V$, we must have $\pmu(\pi^*V')\leq\pmu(\pi^*V)$, that is 
\[\frac{nd'+n\sum_{p\in D}\sum_{i=1}^n\alpha_i(p)}{n'}\leq\frac{nd+n\sum_{p\in D}\sum_{i=1}^n\alpha_i(p)}{n}.\] But this is equivalent to 
\[\frac{d'+\sum_{p\in D}\sum_{i=1}^n\alpha_i(p)}{n'}\leq\frac{d+\sum_{p\in D}\sum_{i=1}^n\alpha_i(p)}{n},\] that is, to
$\pmu(V')\leq\pmu(V)$, proving semistability of $(V,\vp)$.
\endproof

So the push-forward gives a map from $\bigsqcup_{\varpi_n\in S_n^{|D|}}\cM_d^{\bm\a_\gamma(\varpi_n)}(\C^*)\cong T^*\Pic^d(X_\gamma)\times S_n^{|D|}$ to the moduli space $\cM_d^{\bm\a}(\GL(n,\C))$ of degree $d$, strongly parabolic $\GL(n,\C)$-Higgs bundles, with parabolic type $\bm\a$. If we want the determinant to be $\Lambda$, we have to restrict this map to the subspace $T^*\Nm_\pi^{-1}(\Lambda)\times S_n^{|D|}$ if $n\geq 3$ is prime (or to $T^*\Nm_\pi^{-1}(\Lambda L_\gamma)\times S_2^{|D|}$ if $n=2$), where  $T^*\Nm_\pi^{-1}(\Lambda)$ denotes the cotangent bundle to $\Nm_\pi^{-1}(\Lambda)$.  
Indeed, $\det(\pi_*F)\cong\Nm_\pi(F) L_\gamma^{-n(n-1)/2}$, so if $n\geq 3$ is prime, 
$\Nm_\pi(F)\cong\Lambda$ (and if $n=2$, $\Nm_\pi(F)\cong\Lambda L_\gamma$). Thus $(\pi_*F,\pi_*\phi)$ is such that $\det(\pi_*F)\cong\Lambda$ and $\tr(\pi_*\phi)=0$ if and only if $(F,\phi)\in T^*\Nm_\pi^{-1}(\Lambda)\times S_n^{|D|}$, if $n\geq 3$ prime, or $(F,\phi)\in T^*\Nm_\pi^{-1}(\Lambda L_\gamma)\times S_2^{|D|}$ if $n=2$. 
In any case, $\Nm_\pi^{-1}(\Lambda)$ is a torsor for $\ker(\Nm_\pi)$ thus $T^*\Nm_\pi^{-1}(\Lambda)$ is also a torsor for $T^*\ker(\Nm_\pi)$. In turn, from Proposition \ref{prop:connectedcompKerNm-deg=d}, each connected component of $T^*\ker(\Nm_\pi)$ is a torsor for $T^*\Prym_\pi(X_\gamma)$ which is isomorphic to $\Prym_\pi(X_\gamma)\times\C^{(n-1)(g-1)}$, since the Prym is an abelian variety of dimension $(n-1)(g-1)$, by \eqref{eq:dimPrym}. Hence 
\[T^*\Nm_\pi^{-1}(\Lambda)\cong\Nm_\pi^{-1}(\Lambda)\times\C^{(n-1)(g-1)}.\]

Therefore, Proposition \ref{prop:pushforward-semistable} gives a map
\begin{equation}\label{eq;pushforward}
\pi_*:\Nm_\pi^{-1}(\Lambda)\times\C^{(n-1)(g-1)}\times S_n^{|D|}\to\cM,
\end{equation}
with the obvious modification when $n=2$.

Now we can describe the locus $\cM^\gamma$ of points in $\cM$ fixed by a non-trivial element $\gamma\in\Gn$. This locus is going to be the image of $\pi_*$, which is isomorphic to the quotient of its domain by a natural action of the Galois group of $\pi:X_\gamma\to X$.

\begin{theorem}\label{thm:fixedpointlocus1}
Let $\bm\alpha$ be any generic parabolic type and
let $n\geq 3$ be prime. For every $\gamma\in\Gn\setminus\{e\}$, the map \eqref{eq;pushforward} induces an isomorphism 
\[\cM^\gamma\simeq\big(\Nm_\pi^{-1}(\Lambda)\times\C^{(n-1)(g-1)}\times S_n^{|D|}\big)/\Z_n,\]
with $\Z_n$ acting diagonally, by pullback on $T^*\Nm_\pi^{-1}(\Lambda)$ and cyclically on each factor of $S_n^{|D|}$.
If $n=2$, replace $\Nm_\pi^{-1}(\Lambda)$ by $\Nm_\pi^{-1}(\Lambda L_\gamma)$.
\end{theorem}
\proof
Let $\gamma\in\Gn\setminus\{e\}$ and $(V,\vp)$ represent a point in $\cM$. Since $(V,\vp)$ is stable then its only parabolic endomorphisms are the scalars \cite[(3.3)]{thaddeus:2002}, hence the same argument as in Proposition 2.6 of \cite{narasimhan-ramanan:1975} shows that $V\cong V\otimes L_\gamma$ if and only if isomorphic to the push-forward of a line bundle $F$ over $X_\gamma$,
\[V\cong\pi_*F.\] Moreover, the isomorphism $V\cong V\otimes L_\gamma$ is given by 
\begin{equation}\label{eq:isomfixed}
\pi_*(\Id_F\otimes \lambda):\pi_*F\xrightarrow{\cong}\pi_*F\otimes L_\gamma
\end{equation}
where we recall that $\lambda:\cO_{X_\gamma}\to\pi^*L_\gamma$ denotes the tautological section.
From here one sees that the Higgs fields on $V\cong\pi_*F$ which are compatible under the isomorphism \eqref{eq:isomfixed} are the ones which are push-forward of Higgs fields on $F$.

Consider now the parabolic structure on $V\cong\pi_*F$, 
\begin{equation}\label{eq:parabstru-fixed1}
V_p=V_{p,1} \supsetneq V_{p,2} \supsetneq\cdots \supsetneq V_{p,n}\supsetneq \{0\},\qquad 0\leq \a_1(p)< \cdots <\a_n(p) < 1
\end{equation}
over each $p\in D$, so that the one on $V\otimes L_\gamma$ is
\begin{equation}\label{eq:parabstru-fixed2}
V_p\otimes L_\gamma=V_{p,1}\otimes L_{\gamma,p} \supsetneq V_{p,2}\otimes L_{\gamma,p} \supsetneq \cdots \supsetneq V_{p,n}\otimes L_{\gamma,p}\supsetneq \{0\},\qquad 0\leq \a_1(p)< \cdots <\a_n(p) < 1.
\end{equation}
 For $p\in D$, let $\pi^{-1}(p)=\{q_1,\ldots,q_n\}$. Then the isomorphism \eqref{eq:isomfixed} over $p$ is 
\begin{equation}\label{eq:isom-fixed-point-divisor}
V_p=(\pi_*F)_p=\bigoplus_{i=1}^nF_{q_i}\xrightarrow{\bigoplus_{i=1}^n(\Id_{F_{q_i}}\otimes \lambda_{q_i})}\bigoplus_{i=1}^nF_{q_i}\otimes L_{\gamma,p}=V_p\otimes L_{\gamma,p}.
\end{equation} Since $q_i\neq q_j$ then also $\lambda(q_i)\neq \lambda(q_j)$ for $i\neq j$. Hence the only non-trivial subspaces $0\neq V'_p$ of $V_p$ which are preserved under \eqref{eq:isom-fixed-point-divisor} those of the form 
\begin{equation}\label{eq:subspace-filtr-fixed}
V'_p=\bigoplus_{j=1}^{\dim(V'_p)}F_{q_{i_j}}.
\end{equation}
So the isomorphism \eqref{eq:isomfixed} respects the filtrations \eqref{eq:parabstru-fixed1} and \eqref{eq:parabstru-fixed2} if and only if each $V_{p,i}$ is of the form \eqref{eq:subspace-filtr-fixed}.

So, after providing an element $\varpi_n\in S_n^{|D|}$, this parabolic structure of $V$ determines a parabolic structure on $F$ over $D_\gamma=\pi^{-1}(D)$, by reversing the construction carried in \eqref{eq:filt-perm1} and \eqref{eq:filt-perm2}.

The conclusion is that $(V,\vp)\in\cM^\gamma$, i.e., $(V,\vp)\cong(V\otimes L_\gamma,\vp\otimes\Id_{L_\gamma})$ if and only if $(V,\vp)\cong\pi_*((F,\phi),\varpi_n)$ for some $((F,\phi),\varpi_n)\in \Nm^{-1}(\Lambda)\times\C^{(n-1)(g-1)}\times S_n^{|D|}$ (with the obvious modification if $n=2$).

It turns out that there are redundancies coming precisely from the action of the Galois group $\Z_n$ of  $\pi:X_\gamma\to X$. 
Then clearly, for any $j=0,\ldots, n-1$, we have $\pi_*(F,\phi)\cong\pi_*(\xi^{j\ast}F,\xi^{j\ast}\phi)$ as (non-parabolic) Higgs bundles and these are the only redundancies. 
Now we take into account the parabolic structure. A parabolic Higgs bundle is defined by $((F,\phi),\varpi_n)\in \Nm_\pi^{-1}(\Lambda)\times\C^{(n-1)(g-1)}\times S_n^{|D|}$. Consider the action of the Galois group $\Z_n$ given by
\begin{equation}\label{eq:action of Galois group1}
\xi^j\cdot((F,\phi),\varpi_n)=((\xi^{j\ast}F,\xi^{j\ast}\phi),\xi^{-j}\cdot\varpi_n)
\end{equation}
where each $\xi^j\in\Z_n$ acts diagonally on $\varpi_n$ such that on each factor $\varpi_n(p)$ it acts as a cyclic permutation of length $i$ (hence acts freely). Precisely, $\xi^j$ acts diagonally on $\varpi_n$ as
\begin{equation}\label{eq:action of Galois group2}
\xi^{-j}\cdot \varpi_n(p)=a_{n-j+1}(p)\ \dots\ a_n(p)\ a_1(p)\ \dots\ a_{n-j}(p),
\end{equation} for each $p\in D$. Thus the orbit of $\varpi_n$ is given by the set of all the $n^{|D|}$ permutations which differ from the given one at each point by a cyclic permutation.

It is easy to check, by following again the construction  in \eqref{eq:filt-perm1} and \eqref{eq:filt-perm2}, that two elements of $\Nm_\pi^{-1}(\Lambda)\times\C^{(n-1)(g-1)}\times S_n^{|D|}$ give rise to isomorphic parabolic Higgs bundles if and only if they are in the same orbit under $\Z_n$. In other words, $(\pi_*F,\pi_*\phi,\varpi_n)$ and $(\pi_*\xi^{j\ast}F,\pi_*\xi^{j\ast}\phi,\xi^{-j}\cdot\varpi_n)$ determine isomorphic parabolic Higgs bundles, for each $j=0,\ldots,n-1$, and that is the only way one can obtain isomorphic parabolic Higgs bundles under our construction.
We conclude that 
\[\cM^\gamma\simeq\big(\Nm_\pi^{-1}(\Lambda)\times\C^{(n-1)(g-1)}\times S_n^{|D|}\big)/\Z_n\]
with $\Z_n$ acting diagonally as in \eqref{eq:action of Galois group1} and \eqref{eq:action of Galois group2}. 
\endproof

\begin{remark}
  As mentioned in Remark~\ref{rem:alpha-independent-PGLn}, our
  description of $\mathcal{M}^\gamma$ implies that the corresponding
  parabolic Higgs bundles are $\bm\alpha$-semistable for any value of
  $\bm\alpha$.
\end{remark}

In particular, it follows from this theorem that, for any non-trivial $\gamma\in\Gn$,
\begin{equation}\label{eq:dimfixedlocus}
\dim(\cM^\gamma)=2(n-1)(g-1).
\end{equation}

It turns out that the parabolic structure will now make our life
easier by allowing a slightly different description of the fixed point
locus $\cM^\gamma$, from which the calculation of the $\Gn$-invariant
$E$-polynomial $E(\cM^\gamma)^{\Gn}$ is simpler than in the
non-parabolic case.

First, choose a section 
\begin{equation}\label{eq:section s}
s:S_n^{|D|}/\Z_n\to S_n^{|D|}
\end{equation} of the projection $S_n^{|D|}\to S_n^{|D|}/\Z_n$. Of course it corresponds to the choice of a representative of each class in $S_n^{|D|}/\Z_n$. There is no canonical choice of such $s$, but all of them are obviously algebraic. 

Recall that $\Gn$ acts on $\cM^\gamma$, by acting trivially on the Higgs field and on the weights and by pullback and tensor product on the factor $\Nm_\pi^{-1}(\Lambda)$ of $\cM^\gamma$; cf. \eqref{eq:actGnkerNm}. 

\begin{proposition}\label{prop:fixedpointlocus2}
There is a $\Gn$-equivariant isomorphism (depending on the choice of the section $s$ in \eqref{eq:section s})
\[\cM^\gamma\simeq \Nm_\pi^{-1}(\Lambda)\times\C^{(n-1)(g-1)}\times (S_n^{|D|}/\Z_n),\]
where $\Gn$ acts on the first factor as stated in
\eqref{eq:actGnkerNm} and trivially on the other two factors.
\end{proposition}
\proof
From Theorem \ref{thm:fixedpointlocus1}, we know that $\cM^\gamma\simeq\big(\Nm_\pi^{-1}(\Lambda)\times\C^{(n-1)(g-1)}\times S_n^{|D|}\big)/\Z_n$.
Consider the map 
\[f_s:\Nm_\pi^{-1}(\Lambda)\times\C^{(n-1)(g-1)}\times (S_n^{|D|}/\Z_n)\to\big(\Nm_\pi^{-1}(\Lambda)\times\C^{(n-1)(g-1)}\times S_n^{|D|}\big)/\Z_n\] defined by 
\[f_s(F,\phi,[\varpi_n])=[(F,\phi,s([\varpi_n]))].\] 
 Its inverse $g_s$ is defined as follows. Take $[(F,\phi,\varpi_n)]\in (\Nm_\pi^{-1}(\Lambda)\times\C^{(n-1)(g-1)}\times S_n^{|D|})/\Z_n$. Since $\Z_n$ acts freely on $S_n^{|D|}$, there is a unique $i$ such that $\xi^i\cdot\varpi_n=s([\varpi_n])$. Then $[(F,\phi,\varpi_n)]=[(\xi^i\cdot F,\xi^i\cdot\phi,\xi^i\cdot\varpi_n)]$, and hence take 
\[g_s([(F,\phi,\varpi_n)])=(\xi^i\cdot F,\xi^i\cdot\phi,[\varpi_n]).\] It is clear that indeed $g_s=f_s^{-1}$.
It is clear that both $f_s$ and its inverse are algebraic, yielding the stated isomorphism.

To see that it is $\Gn$-equivariant, is just a matter of noticing that, for each $\delta\in\Gn$,
\[\delta\cdot f_s(F,\phi,[\varpi_n])=[F\otimes\pi^*L_\delta,\phi,\varpi_n]=f_s(\delta\cdot(F,\phi,[\varpi_n]))\]
because $\Z_n$ acts trivially on $\pi^*L_\delta$.
\endproof

The action of the Galois group $\Z_n$ on the product $\Nm_\pi^{-1}(\Lambda)\times\C^{(n-1)(g-1)}\times S_n^{|D|}$ can therefore be absorbed in the $S_n^{|D|}$ factor.

\subsection{Calculation of the polynomial}

\begin{proposition}\label{prop:Gamma-invpart}
For any non-trivial $\gamma\in\Gn$, we have the following isomorphism regarding the $\Gn$-invariant part of $H_c^*(\cM^\gamma,\C)$:
\[H_c^*(\cM^\gamma,\C)^{\Gn}\cong H^*(\Prym_\pi(X_\gamma),\C)\otimes H_c^*(\C^{(n-1)(g-1)},\C)\otimes  H^*(S_n^{|D|}/\Z_n,\C).\]
\end{proposition}
\proof
By Proposition \ref{prop:fixedpointlocus2}, we can consider the $\Gn$-action on $\Nm_\pi^{-1}(\Lambda)\times\C^{(n-1)(g-1)}\times (S_n^{|D|}/\Z_n)$, where $\Gn$ acts trivially on the second and third factors, hence the corresponding cohomologies are $\Gn$-invariant.
It then suffices to prove that $H^*(\Nm_\pi^{-1}(\Lambda),\C)^{\Gn}\cong H^*(\Prym_\pi(X_\gamma),\C)$.

Consider the symplectic basis of $\Gn$ given by \eqref{eq:basis}. By Proposition \ref{prop:Gamma-action on components of Nm^-1}, the subgroup generated by $\delta_0$ acts freely and transitively on $\pi_0(\Nm_\pi^{-1}(\Lambda))$, while any $\delta\notin\langle\delta_0\rangle$ acts trivially on these components.
Write the decomposition of $\Nm_\pi^{-1}(\Lambda)$ into connected components as 
\begin{equation}\label{eq:decomp of N}
\Nm_\pi^{-1}(\Lambda)=N_1\sqcup\cdots\sqcup N_n,
\end{equation}
where the indices are chosen so that if $x\in N_i$ then $\delta_0(x)\in N_{i+1}$ (where we assume $n+1=1$).
Each $N_i$ is a torsor for the Prym variety of $X_\gamma$, hence their cohomologies are the same.

Take a cohomology class in $H^k(\Nm_\pi^{-1}(\Lambda),\C)$  represented by a $k$-form $\omega\in\mathcal{A}^*(\Nm_\pi^{-1}(\Lambda),\C)$.
Write 
\begin{equation}\label{eq:decomp of w}
\omega=(\omega_1,\ldots,\omega_i,\ldots,\omega_n)
\end{equation}
according to \eqref{eq:decomp of N}, where each $\omega_i$ represents a cohomology class in $H^k(N_i,\C)$.
The action of $\delta\in\Gn$ on the cohomology class represented by $\omega$ is given by pullback
\begin{equation}\label{eq:Gn-action on forms with twisted coeff}
\delta\cdot\omega=\delta^*\omega.
\end{equation}
So the decomposition of $\delta_0\cdot\omega$ in \eqref{eq:decomp of N} is given by 
\begin{equation}\label{eq:decomp of delta0 w}
\delta_0\cdot\omega=(\delta_0^*\omega_2,\ldots,\delta_0^*\omega_{i+1},\ldots,\delta_0^*\omega_1).
\end{equation}
By \eqref{eq:decomp of w} and \eqref{eq:decomp of delta0 w} we see that the class represented by $\omega$ is invariant by the subgroup $\langle\delta_0\rangle$ if and only if the forms $\omega_i$ are such that $\omega_i=\delta_0^*\omega_{i+1}$, that is, if and only if 
$\omega$ is given by 
\begin{equation}\label{eq:w invariant}
\omega=((\delta_0^{n-1})^*\omega_n,\ldots,(\delta_0^{n-i})^*\omega_n,\ldots,\omega_n).
\end{equation}
Notice that this makes sense because $\delta_0$ has order $n$.

Consider now an element $\delta\in\Gn$ which is not in the subgroup generated by $\delta_0$. Then $\delta$ preserves the connected components $N_i$ of $\Nm_\pi^{-1}(\Lambda)$, acting hence on each $H^*(N_i,\C)$. Each $N_i$ is a torsor for the Prym of $X_\gamma$ and $\delta$ acts on $N_i$  by translations by an element of the Prym:
\[\delta\cdot M=M\otimes\pi^*L_\delta.\] 
To see that indeed $\pi^*L_\delta\in\Prym_\pi(X_\gamma)$, note first that it is in the kernel of $\Nm_\pi$. Hence it is of the form $\pi^*L_\delta=F^{-1}\otimes\xi^*F$, with $\xi^{\deg(F)}=\langle\delta,\gamma\rangle$, by Proposition \ref{prop:Weilpairing-degree}. But $\langle\delta,\gamma\rangle=1$ i.e. $\deg(F)=0$ and $\pi^*L_\delta\in\Prym_\pi(X_\gamma)$ by Proposition \ref{prop:connectedcompKerNm-deg}. Since  $\Prym_\pi(X_\gamma)$ is an abelian variety, every class in $H^*(\Prym_\pi(X_\gamma),\C)$ contains a unique representative which is invariant under translations. This property goes through  $H^*(N_n,\C)\cong H^*(\Prym_\pi(X_\gamma),\C)$ considering the torsor structure of $N_n$. This means that we can assume that the form $\omega_n$ in \eqref{eq:w invariant} is invariant under translations, so is $\delta$-invariant, i.e., $\delta^*\omega_n=\omega_n$.
Hence the action \eqref{eq:Gn-action on forms with twisted coeff} of $\delta$ on $\omega$ given by \eqref{eq:w invariant} is 
\[\delta\cdot\omega=((\delta_0^{n-1})^*\omega_n,\ldots,(\delta_0^{n-i})^*\omega_n,\ldots,\omega_n).\]  

We thus conclude that $H^*(\Nm_\pi^{-1}(\Lambda),\C)^{\Gn}$ is given precisely by the classes represented by the forms of type \eqref{eq:w invariant}. Mapping those to $[\omega_n]$ gives an isomorphism with $H^*(N_n,\C)$, hence also with $H^*(\Prym_\pi(X_\gamma),\C)$.
\endproof

Now we can finally compute the sum of the stringy $E$-polynomial of $\cM/\Gn$ corresponding to non-trivial elements of $\Gn$. 

\begin{proposition}\label{prop:stringyE-poly}
For any $n$ prime, the following holds:
\[\sum_{\gamma\neq e}E(\cM^\gamma)^{\Gn}(uv)^{F(\gamma)}=\frac{n^{2g}-1}{n}(n!)^{|D|}(uv)^{(n^2-1)(g-1)+|D|n(n-1)/2}((1-u)(1-v))^{(n-1)(g-1)}.\]
\end{proposition}
\proof
By Proposition \ref{prop:Gamma-invpart}, and since $E(\C^{(n-1)(g-1)})=(uv)^{(n-1)(g-1)}$,
\[E(\cM^\gamma)^{\Gn}=(uv)^{(n-1)(g-1)}E(\Prym_\pi(X_\gamma))E(S_n^{|D|}/\Z_n),\]
 for each $\gamma\in\Gn\setminus\{e\}$. 
 
 The polynomial $E(S_n^{|D|}/\Z_n)$ is just the constant $\frac{1}{n}(n!)^{|D|}$, i.e., the number of elements of the space $S_n^{|D|}/\Z_n$.
 
Being an abelian variety, the cohomology of the Prym of $X_\gamma$ is the alternating algebra on $H^1(\Prym_\pi(X_\gamma),\C)=H^{0,1}(\Prym_\pi(X_\gamma))\oplus H^{1,0}(\Prym_\pi(X_\gamma))$. Write $V=H^{0,1}(\Prym_\pi(X_\gamma))$ and note that $\dim(V)=\dim(\Prym_\pi(X_\gamma))=(n-1)(g-1)$.
Thus $H^k(\Prym_\pi(X_\gamma),\C)=\Lambda^k(V\oplus\bar V)$,
therefore 
\[H^{p,q}(\Prym_\pi(X_\gamma))=\Lambda^pV\otimes\Lambda^q\bar V,\]
whose dimension is $\binom{(n-1)(g-1)}{p}\binom{(n-1)(g-1)}{q}$.
Hence
\begin{equation*}
\begin{split}
E(\Prym_\pi(X_\gamma))&=\sum_{p,q=0}^{(n-1)(g-1)}(-1)^{p+q}\binom{(n-1)(g-1)}{p}\binom{(n-1)(g-1)}{q}u^pv^q\\
&=((1-u)(1-v))^{(n-1)(g-1)}.
\end{split}
\end{equation*}

We are now left to the computation of the fermionic shift $F(\gamma)$ as defined in \eqref{eq:fermionicshift}. From \eqref{eq:fermionic=rkN/2}, we know that $F(\gamma)=\dim(N_p\cM^\gamma)/2$, but from \eqref{eq:dim} and \eqref{eq:dimfixedlocus}, we conclude that
\begin{equation}\label{eq:comput of fermionic}
F(\gamma)=n(n-1)(g-1+|D|/2).
\end{equation}

Therefore, for each $\gamma\neq e$,
\[E(\cM^\gamma)^{\Gn}(uv)^{F(\gamma)}=\frac{1}{n}(n!)^{|D|}(uv)^{(n-1)(g-1)+n(n-1)(g-1+|D|/2)}((1-u)(1-v))^{(n-1)(g-1)}.\]
This is independent of $\gamma\in\Gn\setminus\{e\}$, thus summing up this expression for all non-trivial elements of $\Gn$, yields
\[
\sum_{\gamma\neq e}E(\cM^\gamma)^{\Gn}(uv)^{F(\gamma)}=\frac{n^{2g}-1}{n}(n!)^{|D|}(uv)^{(n^2-1)(g-1)+|D|n(n-1)/2}((1-u)(1-v))^{(n-1)(g-1)},
\]
as claimed.
\endproof

\vspace{1cm}

\noindent 
\textbf{Peter B. Gothen} \\
      Centro de Matemática da Universidade do Porto, CMUP\\
      Faculdade de Ciências, Universidade do Porto\\
      Rua do Campo Alegre 687, 4169-007 Porto, Portugal\\ 
      \url{www.fc.up.pt}\\
      email: pbgothen@fc.up.pt

\vspace{1cm}

\noindent
      \textbf{André G. Oliveira} \\
      Centro de Matemática da Universidade do Porto, CMUP\\
      Faculdade de Ciências, Universidade do Porto\\
      Rua do Campo Alegre 687, 4169-007 Porto, Portugal\\ 
      \url{www.fc.up.pt}\\
      email: andre.oliveira@fc.up.pt

\vspace{.2cm}
\noindent
\textit{On leave from:}\\
 Departamento de Matemática, Universidade de Trás-os-Montes e Alto Douro, UTAD \\
Quinta dos Prados, 5000-911 Vila Real, Portugal\\ 
\url{www.utad.pt}\\
email: agoliv@utad.pt

\end{document}